\documentclass[pdflatex,sn-mathphys-num]{sn-jnl}

\usepackage{graphicx}%
\usepackage{multirow}%
\usepackage{amsmath,amssymb,amsfonts}%
\usepackage{amsthm}%
\usepackage{mathrsfs}%
\usepackage[title]{appendix}%
\usepackage{xcolor}%
\usepackage{textcomp}%
\usepackage{manyfoot}%
\usepackage{booktabs}%
\usepackage{algorithm}%
\usepackage{algorithmicx}%
\usepackage{algpseudocode}%
\usepackage{listings}%
\usepackage{tikz}
\usetikzlibrary{arrows.meta,positioning,fit,backgrounds,calc}

\newcommand{\C}{\mathbb{C}}
\newcommand{\R}{\mathbb{R}}
\newcommand{\Z}{\mathbb{Z}}
\newcommand{\N}{\mathbb{N}}
\newcommand{\Cau}{\mathcal{C}}
\newcommand{\Lop}{\mathcal{L}}
\newcommand{\Lcx}{\mathfrak{L}}
\newcommand{\Gw}{\Gamma_{\theta,c}}

\newcommand{\dist}{\operatorname{dist}}
\newcommand{\Real}{\operatorname{Re}}

\newcommand{\bd}{\boldsymbol}
\newcommand{\Lq}{\mathcal{L}_{q}}
\newcommand{\Cint}{\mathcal{C}_{\mathrm{int}}}
\newcommand{\Eop}{\mathcal{E}_{p,q,\theta,c}}
\newcommand{\Ahol}{\mathcal{A}}
\newcommand{\Log}{\operatorname{Log}}

\theoremstyle{thmstyleone}
\newtheorem{theorem}{Theorem}
\newtheorem{lemma}[theorem]{Lemma}
\newtheorem{proposition}[theorem]{Proposition}%
\newtheorem{assumption}[theorem]{Assumption}
\theoremstyle{thmstyletwo}

\newtheorem{remark}{Remark}
\newtheorem{corollary}[theorem]{Corollary}

\theoremstyle{thmstylethree}
\newtheorem{definition}{Definition}

\begin{document}

\title[Finite-sheeted Cauchy operator at rational corners]
{Finite-sheeted Cauchy operator at rational corners}

\author*[1]{\fnm{Louis Shuo} \sur{Wang}}
\email{wang.s41@northeastern.edu}

\affil[1]{\orgdiv{Department of Mathematics},
\orgname{Northeastern University},
\orgaddress{
\city{Boston},
\state{MA},
\postcode{02115},
\country{USA}
}}

\abstract{We study Cauchy singular integral operators on planar wedges whose opening
angle is a rational multiple of $\pi$. For $\theta=p\pi/q$, the covering
$w=\zeta^q$ yields an exact finite-sheeted factorization of the wedge Cauchy
transform into $2q$ interval Cauchy transforms with explicit algebraic
recombination coefficients. The factorization is formulated on weighted
conormal H\"older spaces. We prove that the lifting operator preserves conormal
order, lowers the H\"older exponent from $\beta$ to $\beta/q$, and has sharp
$\ell^1$ sheet norm $q$. Combining this operator factorization with a Mellin
model for interval Cauchy transforms, we derive a mode-by-mode propagation rule
for polyhomogeneous endpoint expansions. Nonresonant powers preserve their
logarithmic order, while integer exponents raise it by one. The results also
give a local singular decomposition for Cauchy operators on piecewise analytic
curves with rational corner angles.}


\maketitle

\maketitle

\section{Introduction}
\begin{figure}[htbp]
\centering
\includegraphics[width=0.85\linewidth]{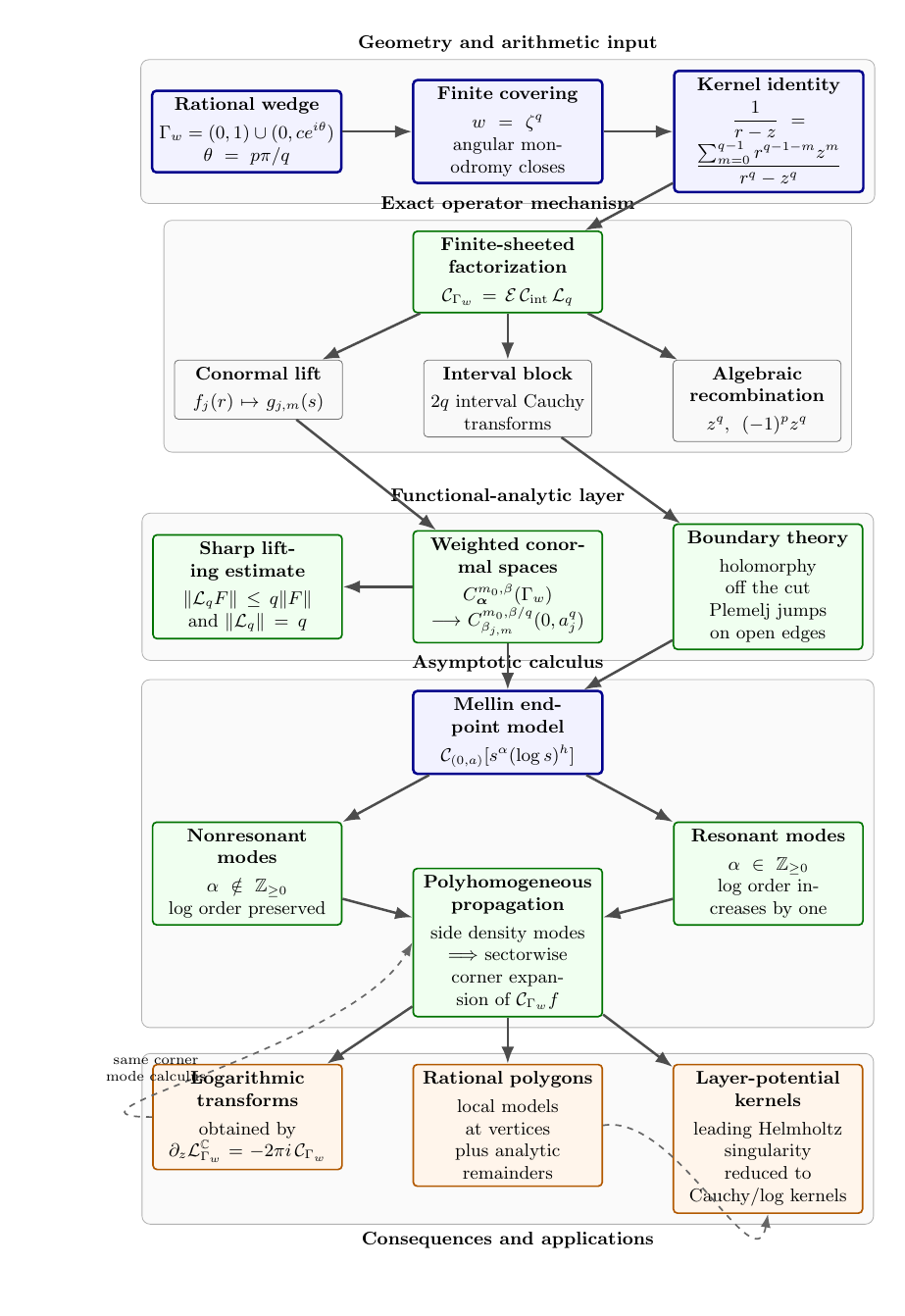}
\caption{Logical framework of the paper. A rational corner closes under the
finite covering $w=\zeta^q$, which converts the wedge Cauchy operator into a
finite block of interval Cauchy transforms. The conormal lifting estimate gives
the functional-analytic control, while the Mellin endpoint model gives the
mode-by-mode propagation rule. Logarithmic transforms, rational-polygon
localization, and Helmholtz layer-potential kernels are then obtained as
consequences of the same finite-sheeted corner calculus.}
\label{fig:logical-framework}
\end{figure}

In this paper, we consider \begin{equation}\label{eq:wedge}
  \Gw=(0,1)\cup(0,ce^{i\theta}),\qquad
  \theta=\frac{p\pi}{q},\quad (p,q)=1,\quad 0<\theta<2\pi,\quad c>0,
\end{equation}
with both sides oriented away from the vertex $0$. For an integrable density
$f$ the Cauchy transform is
\begin{equation}\label{eq:cauchy-def}
  \Cau_{\Gw}f(z)=\frac{1}{2\pi i}\int_{\Gw}\frac{f(\zeta)}{\zeta-z}\,d\zeta,
  \qquad z\in\C\setminus\Gw ,
\end{equation}
and we write the side restrictions as
\begin{equation}\label{eq:sides}
  f_1(r)=f(r),\quad 0<r<1,\qquad f_2(r)=f(re^{i\theta}),\quad 0<r<c .
\end{equation}
For a finite segment $(0,a)\subset\R$ and $g\in L^1(0,a)$ we use the interval
Cauchy and logarithmic transforms
\begin{equation}\label{eq:interval}
  \Cau_{(0,a)}[g](\xi)=\frac{1}{2\pi i}\int_0^a\frac{g(r)}{r-\xi}\,dr,
  \qquad
  \Lcx_{(0,a)}[g](\xi)=\int_0^a\log(r-\xi)\,g(r)\,dr,
\end{equation}
for $\xi\in\C\setminus[0,a]$. We write $\Ahol(U)$ for the holomorphic functions
on an open set $U\subset\C$. The vertex $0$ is the only geometric singularity of
the contour; it is the source of the algebraic branch behaviour underlying the
failure of smooth-boundary estimates for the associated singular integral
operators \cite{kerzman1978cauchy,bolt2015kerzman}.

\subsection{The finite-sheeted factorization}
The rationality of $\theta$ resolves the vertex by the covering $w=\zeta^{q}$.
The mechanism is algebraic: from $r^{q}-z^{q}=\prod_{k}(r-z\omega^{k})$,
$\omega=e^{2\pi i/q}$, one has the geometric-sum identity
\begin{equation}\label{eq:geo}
  \frac{1}{r-z}=\frac{1}{r^{q}-z^{q}}\sum_{m=0}^{q-1}r^{q-1-m}z^{m},
\end{equation}
which, applied to each side after $s=r^{q}$, yields the decomposition
(Theorem~\ref{thm:decomposition})
\begin{equation}\label{eq:main}
  \Cau_{\Gw}f(z)
  =\frac1q\sum_{m=0}^{q-1}z^{q-1-m}\,\Cau_{(0,1)}[g_{1,m}](z^{q})
  +\frac1q\sum_{m=0}^{q-1}\bigl(e^{-i\theta}z\bigr)^{q-1-m}\,
     \Cau_{(0,c^{q})}[g_{2,m}]\bigl((-1)^{p}z^{q}\bigr),
\end{equation}
with explicit lifted densities
\begin{equation}\label{eq:gjm}
  g_{j,m}(s)=f_j\!\bigl(s^{1/q}\bigr)\,s^{(m+1)/q-1},
  \qquad j=1,2,\ m=0,\dots,q-1 .
\end{equation}
This paper organizes \eqref{eq:main} as an honest operator factorization and
builds the analytic theory it supports. We package \eqref{eq:main} as
\begin{equation}\label{eq:factorization}
  \Cau_{\Gw}=\Eop\,\Cint\,\Lq,
\end{equation}
where $\Lq$ is the conormal lifting operator $F=(f_1,f_2)\mapsto(g_{j,m})$,
$\Cint=\operatorname{diag}\bigl(\Cau_{(0,1)},\dots,\Cau_{(0,c^{q})}\bigr)$ is the
diagonal block of interval Cauchy operators, and $\Eop$ is the algebraic
evaluation--recombination operator (\S\ref{sec:factor}). The terminology
``finite-sheeted'' refers throughout to the algebraic organization of the kernel
by the covering $w=\zeta^{q}$, not to a symmetric extension of the density to
all sheets: only the radial variable is lifted, side by side, and the density is
never continued to a globally defined function on the covering. This
distinguishes \eqref{eq:main} from the generic rational-power and
Hilbert-transform change-of-variable literature, in which transforms of $f$ are
related to transforms of $f(x^{r})$ and rational powers lead to Cauchy
transforms on Riemann surfaces \cite{king2009hilbert,olver2011computing}; Olver's
Riemann--Hilbert framework, including reductions to intervals and half-lines
\cite{olver2011computing}, is the closest precedent. The contribution here is the
explicit single-branch, orientation-aware factorization \eqref{eq:factorization}
and the operator-theoretic and asymptotic calculus it supports.

\subsection{Main results}
The paper proves six families of results.

\emph{(I) Operator factorization (\S\ref{sec:factor}).} The identity
\eqref{eq:factorization} holds between conormal weighted H\"older spaces on
$\Gw$ and $\Ahol(\C\setminus\Gw)$ (Theorem~\ref{thm:factorization}). The
contrasting symmetric-star collapse (Lemma~\ref{lem:star}) clarifies why the
wedge factorization is single-branch.

\emph{(II) Conormal mapping theory (\S\ref{sec:conormal}).} The lifting operator
$\Lq$ is bounded between conormal weighted H\"older spaces, with the explicit
sheet-count bound
\begin{equation}\label{eq:Lq-bound-intro}
  \|\Lq F\|\le q\,\|F\|,
\end{equation}
independent of $\alpha,\beta,a,m_0$ and sharp, so $\|\Lq\|=q$
(Theorem~\ref{thm:Lq-bounded}). The interval Cauchy operators are holomorphic
off the cut, satisfy a separated evaluation bound, and admit Plemelj boundary
values on compact subintervals (Proposition~\ref{prop:interval-mapping}); the
composite \eqref{eq:factorization} inherits these.

\emph{(III) Polyhomogeneous corner-mode propagation (\S\ref{sec:propagation}).}
For densities with classical conormal expansions
\[
  f_j(r)\sim\sum_{\ell\ge0}\sum_{h=0}^{N_\ell}a_{j,\ell,h}\,
     r^{\alpha_{j,\ell}}(\log r)^{h},\qquad r\downarrow0,
\]
the transform admits a sectorwise vertex expansion
\[
  \Cau_{\Gw}f(z)\sim
  \sum_{j,\ell,h'}A^{\bullet}_{j,\ell,h'}\,z^{\alpha_{j,\ell}}(\log z)^{h'}
  +\sum_{\nu,h'}B^{\bullet}_{\nu,h'}\,z^{\nu}(\log z)^{h'},
\]
with explicit branch-dependent coefficients, in which a nonresonant mode
preserves the top logarithmic power and a resonance $\alpha_{j,\ell}\in\Z_{\ge0}$
raises it by exactly one (Theorem~\ref{thm:propagation}). The engine is the
identity $r^{\alpha}(\log r)^{h}=\partial_\alpha^{h}r^{\alpha}$ applied to a
Mellin model (Lemma~\ref{lem:polyhom-model}).

\emph{(IV) Logarithmic transforms (\S\ref{sec:log}).} The lift transports to a
finite-sheeted logarithmic decomposition, characterized through the
antiderivative relation $\partial_z\Lcx_{\Gw}=-2\pi i\,\Cau_{\Gw}$ together with a
component-dependent branch constant (Theorem~\ref{thm:log}); and we give the
corrected relation between the real logarithmic transform $\Lop_{\Gw}$ and its
complexification $\Lcx_{\Gw}$ (Proposition~\ref{prop:realpart}).

\emph{(V) Helmholtz singular structure (\S\ref{sec:helmholtz}).} Near a rational
polygonal corner the Helmholtz single-, double-, and adjoint double-layer
operators decompose into a rational-wedge corner operator plus a strictly
smoother remainder; the corner double-layer and adjoint operators are
real-linear combinations of boundary values of $\Cau_{\Gw}$ and its conjugate
(Theorem~\ref{thm:helmholtz}, Proposition~\ref{prop:potential-algebra}).

\emph{(VI) Localization on rational polygons (\S\ref{sec:localization}).} On a
piecewise analytic Jordan curve with rational corner angles the Cauchy operator
decomposes as
\[
  \Cau_\Gamma=\sum_{v\in V(\Gamma)}\Cau_v+\Cau_\Gamma^{\mathrm{sm}}+\mathcal R_\Gamma,
\]
where each $\Cau_v$ is a transported finite-sheeted wedge model, $\Cau_\Gamma^{\mathrm{sm}}$
is the smooth-arc Cauchy operator, and $\mathcal R_\Gamma$ is analytic in a
neighbourhood of every vertex (Theorem~\ref{thm:localization}). Consequently the
vertex asymptotics of $\Cau_\Gamma f$ are the direct sum of the wedge expansions
of \S\ref{sec:propagation}, with no new exponents from curvature or inter-vertex
coupling (Corollary~\ref{cor:vertexwise}); the same localization carries the
logarithmic and Helmholtz operators.

A short outlook (\S\ref{sec:outlook}) addresses irrational angles. The novelty is
an exact local-to-global symbolic calculus for the singular part of these
operators at rational-angle corners, not a quadrature or solver.

\subsection{Literature Review}
The Plemelj--Privalov theory of weighted H\"older spaces on arcs is classical
\cite{reyes2003one,muskhelishvili2008singular,gakhov2014boundary,wang2025analysis,wang2025analysis1,pritsker2008find,abreu2012plemelj,yu2026from,abreu2007notion,liu2025bidirectional}, as is the use of conormal/Mellin scales for corner
asymptotics: the Mellin analysis near conical and angular
points can be found in \cite{schrohe2023introduction,liang2025global,schulze2022mellin,anjam2020geometric,maday2019regularity,wang2026algebraic}.
Our propagation theorem (\S\ref{sec:propagation}) is the transform-side
analogue: it tracks how such expansions of the boundary density are carried,
mode by mode, through the Cauchy operator, with an explicit resonance rule. At
the level of boundedness, the deep theory is the $L^{2}$ boundedness of the
Cauchy integral on Lipschitz curves
\cite{david1984boundedness,gao2022rolling,rochberg1984calderon,calderon1977cauchy,wang2026damage,mcintosh2006convolution,wang2025multi,coifman1982integrale,coifman1989two,yu2026pattern,qian2019singular,jones2006square}, on which the method of
layer potentials on nonsmooth domains rests \cite{cai2026optimal,verchota1984layer,costabel1988boundary,mitrea1997method,wang2026breakdown,mitrea2013multi,fabes1977double};
the present paper is complementary, giving not an estimate but an exact local
symbolic calculus for the singular part at a rational corner. Closest in spirit
is the Mellin-transform method for boundary integral equations on curves with
corners, in particular the analysis of the double-layer potential on polygons
\cite{costabel1983normal,yu2026beyond,maz1991boundary,qiao2012single,yu2026rigorous,qiao2018double}; our finite-sheeted factorization makes the
underlying covering $w=\zeta^{q}$ and its branch recombination explicit. The
loss of compactness of the Kerzman--Stein operator $\Cau-\Cau^{*}$ at corners, compact
in the smooth case, is established for piecewise
continuously differentiable curves in \cite{bolt2015kerzman}, where the finite
symmetric wedge furnishes the essential spectrum; the present calculus is the
constructive counterpart, tracking the algebraic modes responsible
(\S\ref{sec:propagation}). Classical boundary integral treatments of corners
proceed by graded meshes and specialized quadrature \cite{kress1989linear,kress1991boundary}; we
instead give an exact local operator factorization, which is complementary
(\S\ref{sec:helmholtz}).

\subsection{Organization}
Section~\ref{sec:factor} establishes the rotation reduction, the single-interval
lift, the symmetric-star contrast, and the operator factorization
\eqref{eq:factor-thm}. Section~\ref{sec:conormal} develops the conormal weighted
H\"older scale, proves the conormal mapping theorem for $\Lq$, and records the
interval Cauchy mapping and the Plemelj jump relations.
Section~\ref{sec:propagation} is the asymptotic core: the Mellin model for pure
powers, its polyhomogeneous extension via differentiation in the exponent, and
the full sectorwise propagation theorem with the resonance rule.
Section~\ref{sec:log} transports the calculus to logarithmic transforms and
corrects the real-part relation. Section~\ref{sec:helmholtz} gives the Helmholtz
singular-operator algebra and the local corner decomposition, with a comparison
to classical corner treatments. Section~\ref{sec:examples} works out the right
angle and checks the resonance rule against exact formulae;
Section~\ref{sec:localization} proves the local-to-global decomposition on
rational polygons; and Section~\ref{sec:outlook} discusses irrational angles.
\section{Finite-sheeted factorization of rational-wedge Cauchy transforms}
\label{sec:factor}

\subsection{Elementary reduction}
The factorization is based on a two-term reduction that is valid for any opening angle.

\begin{proposition}[Rotation reduction]\label{prop:reduction}
Let $\theta\in(0,2\pi)$, $c>0$, and let $f$ have side restrictions
\eqref{eq:sides} with $f_1\in L^1(0,1)$, $f_2\in L^1(0,c)$. Then for every
$z\in\C\setminus\Gw$,
\begin{equation}\label{eq:reduction}
  \Cau_{\Gw}f(z)=\Cau_{(0,1)}[f_1](z)+\Cau_{(0,c)}[f_2]\!\bigl(e^{-i\theta}z\bigr).
\end{equation}
\end{proposition}

\begin{proof}
The first side, $\zeta=r$, $d\zeta=dr$, gives $\Cau_{(0,1)}[f_1](z)$. On the
second, $\zeta=re^{i\theta}$, $d\zeta=e^{i\theta}dr$, so
\[
  \frac{1}{2\pi i}\int_0^c\frac{f_2(r)\,e^{i\theta}}{re^{i\theta}-z}\,dr
  =\frac{1}{2\pi i}\int_0^c\frac{f_2(r)}{r-ze^{-i\theta}}\,dr
  =\Cau_{(0,c)}[f_2]\!\bigl(e^{-i\theta}z\bigr).\qedhere
\]
\end{proof}

Figure~\ref{fig:rotation_reduction} illustrates the elementary rotation reduction, which uncouples the two boundary sides of the wedge into standard intervals on the positive real axis.

\begin{figure}[htbp]
    \centering
    \includegraphics[width=0.85\textwidth]{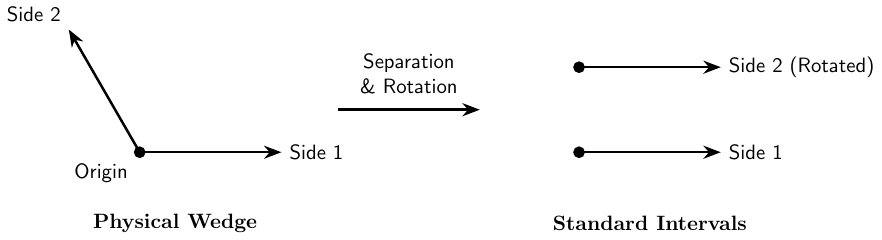}
    \caption{Schematic of the elementary rotation reduction. The physical rational wedge $\Gamma_{\theta,c}$ (left) is geometrically uncoupled into two standard horizontal intervals on the positive real axis (right) via separation and a clockwise rotation of Side 2 by $-\theta$, illustrating the boundary splitting mechanism analyzed in Proposition~\ref{prop:reduction}.}
    \label{fig:rotation_reduction}
\end{figure}

\begin{remark}\label{rem:no-cut}
The right side of \eqref{eq:reduction} is defined whenever
$e^{-i\theta}z\notin(0,c)$, i.e., whenever $z$ lies off the second side; no root
function and no branch cut occur. The cut on $(-\infty,0]$ in the lifted
formulation below is an artifact of the covering and is absent from the honest
reduction.
\end{remark}

\subsection{Lifting one interval transform}
The geometric-sum identity lifts a single interval transform to a finite sum in
$\xi^{q}$.

\begin{lemma}[Single-interval lift]\label{lem:lift}
Let $a>0$, $g\in L^1(0,a)$, and $\xi\in\C$ with $\xi^{q}\notin(0,a^{q})$. Then
\begin{equation}\label{eq:lift}
  \Cau_{(0,a)}[g](\xi)
  =\frac1q\sum_{m=0}^{q-1}\xi^{q-1-m}\,\Cau_{(0,a^{q})}[\tilde g_m](\xi^{q}),
  \qquad
  \tilde g_m(s)=g\!\bigl(s^{1/q}\bigr)\,s^{(m+1)/q-1},
\end{equation}
$s^{1/q}$ the real positive root on $(0,a^{q})$. The right side involves only the
single-valued power $\xi^{q}$ and the polynomial weights $\xi^{q-1-m}$.
\end{lemma}

\begin{proof}
Insert \eqref{eq:geo} (with $r,z$ replaced by $r,\xi$) into
$\Cau_{(0,a)}[g](\xi)=\frac{1}{2\pi i}\int_0^a g(r)(r-\xi)^{-1}dr$:
\[
  \Cau_{(0,a)}[g](\xi)
  =\sum_{m=0}^{q-1}\xi^{m}\,\frac{1}{2\pi i}\int_0^a
   \frac{g(r)\,r^{q-1-m}}{r^{q}-\xi^{q}}\,dr .
\]
Substituting $s=r^{q}$, $ds=q\,r^{q-1}dr$, so $r^{q-1-m}dr=\tfrac1q s^{-m/q}ds$
and $r=s^{1/q}$, gives
$\frac1q\Cau_{(0,a^{q})}[g(s^{1/q})s^{-m/q}](\xi^{q})$. Reindexing
$m\mapsto q-1-m$ turns $\xi^{m}$ into $\xi^{q-1-m}$ and $s^{-m/q}$ into
$s^{(m+1)/q-1}$, giving \eqref{eq:lift}.
\end{proof}

As shown in Figure~\ref{fig:symmetric_star}, the physical rational wedge can be naturally embedded into a fully symmetric star framework.

\begin{figure}[htbp]
    \centering
    \includegraphics[width=0.85\textwidth]{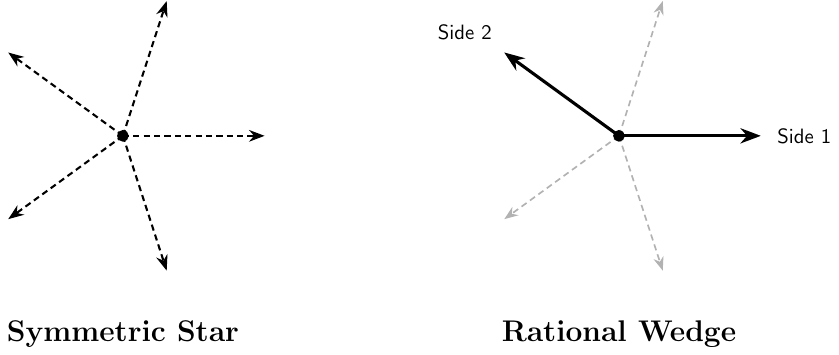}
    \caption{Geometric embedding of a rational physical wedge $\Gamma_{\theta,c}$ with $\theta = 144^\circ$ (right) into a fully symmetric 5-ray star structure (left), where the dashed lines indicate the underlying symmetric skeletal framework.}
    \label{fig:symmetric_star}
\end{figure}

\subsection{The symmetric-star contrast}
Identity \eqref{eq:lift} is the single-branch lift of one physical side. It must
not be confused with the symmetric residue collapse.

\begin{lemma}[Symmetric-star collapse]\label{lem:star}
Let $\Sigma_q=\bigcup_{k=0}^{q-1}(0,e^{2\pi ik/q})$ be the unit $q$-ray star,
each ray oriented outward, carrying the $q$-fold symmetric density equal to
$g_0\in L^1(0,1)$ on every ray. Then, for $z^{q}\notin(0,1)$,
\begin{equation}\label{eq:star}
  \Cau_{\Sigma_q}[g](z)=\Cau_{(0,1)}\!\bigl[g_0(s^{1/q})\bigr](z^{q}).
\end{equation}
\end{lemma}

\begin{proof}
As in Proposition~\ref{prop:reduction}, the $k$-th ray contributes
$\frac{1}{2\pi i}\int_0^1 g_0(r)(r-z\omega^{-k})^{-1}dr$. Summing over $k$ and
using $\sum_{k}(r-z\omega^{-k})^{-1}=q r^{q-1}/(r^{q}-z^{q})$ (logarithmic
derivative of $r^{q}-z^{q}$) gives
$\frac{1}{2\pi i}\int_0^1 g_0(r)\,q r^{q-1}(r^{q}-z^{q})^{-1}dr$; the
substitution $s=r^{q}$ yields \eqref{eq:star}.
\end{proof}

Identity \eqref{eq:star} reduces a fully populated symmetric star to one
interval transform in $z^{q}$; it requires the same density on all $q$ rays. The
wedge \eqref{eq:wedge} populates only two rays of such a star, so it does not
collapse: it is governed by the two-term reduction \eqref{eq:reduction} and its
lift \eqref{eq:main}.

\subsection{The factorization operators}
We now assemble \eqref{eq:main} into operators. Fix $(p,q,\theta,c)$, set
$a_1=1$, $a_2=c$, and let $\bd\alpha=(\alpha_1,\alpha_2)$. The conormal weighted
spaces $C^{m_0,\beta}_{\bd\alpha}(\Gw)$ and $C^{m_0,\beta}_\mu(0,a)$ are defined
in \S\ref{sec:conormal}; here we use only that their elements are integrable
densities.

\begin{definition}[Factorization operators]\label{def:factor-ops}
Define:
\begin{enumerate}
\item the \emph{conormal lifting operator}
\[
  \Lq:\ F=(f_1,f_2)\longmapsto
   G=(g_{1,0},\dots,g_{1,q-1},g_{2,0},\dots,g_{2,q-1}),\qquad
   g_{j,m}(s)=f_j(s^{1/q})\,s^{(m+1)/q-1};
\]
\item the \emph{interval block}
$\Cint=\operatorname{diag}\bigl(\underbrace{\Cau_{(0,1)},\dots,\Cau_{(0,1)}}_{q},
\underbrace{\Cau_{(0,c^{q})},\dots,\Cau_{(0,c^{q})}}_{q}\bigr)$, sending
$G\mapsto(h_{j,m})$ with $h_{j,m}=\Cau_{(0,a_j^{q})}[g_{j,m}]\in
\Ahol(\C\setminus[0,a_j^{q}])$;
\item the \emph{evaluation--recombination operator}
\[
  \Eop:\ (h_{j,m})\longmapsto
  \Bigl[z\mapsto
   \frac1q\sum_{m=0}^{q-1}z^{q-1-m}h_{1,m}(z^{q})
   +\frac1q\sum_{m=0}^{q-1}(e^{-i\theta}z)^{q-1-m}h_{2,m}\bigl((-1)^{p}z^{q}\bigr)
  \Bigr].
\]
\end{enumerate}
\end{definition}

\begin{theorem}[Operator factorization]\label{thm:factorization}
For $(p,q)=1$ and $f$ with $f_1\in L^1(0,1)$, $f_2\in L^1(0,c)$, set
$\varphi_1(z)=z^{q}$, $\varphi_2(z)=(-1)^{p}z^{q}$ and let
$\Sigma_j=\varphi_j^{-1}([0,a_j^{q}])$ be the $q$-ray star on which the $j$-th
side block of $\Eop\Cint\Lq$ is literally defined off. Then
\begin{equation}\label{eq:factor-thm}
  \Cau_{\Gw}=\Eop\,\Cint\,\Lq
  \qquad\text{on } \C\setminus(\Sigma_1\cup\Sigma_2),
\end{equation}
where both sides are holomorphic. Moreover, each side block recombines to a
function holomorphic off a single physical ray---$\Cau_{(0,1)}[f_1]$ off
$[0,1]$ and $\Cau_{(0,c)}[f_2](e^{-i\theta}\,\cdot)$ off the second
side---so the spurious star rays $\Sigma_j\setminus\Gw$ are removable: both
sides of \eqref{eq:factor-thm} extend holomorphically to $\C\setminus\Gw$, and
the identity persists there by analytic continuation.
\end{theorem}

\begin{proof}
By Lemma~\ref{lem:lift} with $(a,g,\xi)=(1,f_1,z)$, the first side block of
$\Eop\Cint\Lq F$ equals $\Cau_{(0,1)}[f_1](z)$ wherever $z^{q}\notin[0,1]$, i.e.
on $\C\setminus\Sigma_1$, and the right-hand side $\Cau_{(0,1)}[f_1]$ is in fact
holomorphic on the larger set $\C\setminus[0,1]$. With
$(a,g,\xi)=(c,f_2,e^{-i\theta}z)$ and
$(e^{-i\theta}z)^{q}=e^{-ip\pi}z^{q}=(-1)^{p}z^{q}$, $a^{q}=c^{q}$, the second
side block equals $\Cau_{(0,c)}[f_2](e^{-i\theta}z)$ on $\C\setminus\Sigma_2$,
holomorphic off the second physical ray. On $\C\setminus(\Sigma_1\cup\Sigma_2)$
their sum is \eqref{eq:reduction}, i.e. $\Cau_{\Gw}f$, giving
\eqref{eq:factor-thm}. Since both sides of the blocks recombine to functions
holomorphic off the physical rays, the common value extends holomorphically
across the spurious rays $\Sigma_j\setminus\Gw$ to all of $\C\setminus\Gw$,
where the identity continues to hold by the identity theorem.
\end{proof}

\begin{theorem}[Finite-sheeted decomposition]\label{thm:decomposition}
With $g_{j,m}$ as in \eqref{eq:gjm}, the decomposition \eqref{eq:main} holds, as
written, for $z$ with $z^{q}\notin[0,1]$ and $(-1)^{p}z^{q}\notin[0,c^{q}]$ (so
that each interval transform on the right, defined on $\C\setminus[0,a]$, is
evaluated off its closed cut), i.e., on $\C\setminus(\Sigma_1\cup\Sigma_2)$. Both
sides continue holomorphically to all of $\C\setminus\Gw$---a larger set in
general, since $\Sigma_1\cup\Sigma_2$ contains up to $2q$ lifted star segments
whereas $\Gw$ consists only of the two physical segments---and the identity
persists there by analytic continuation, the nonphysical star segments being
removable as in Theorem~\ref{thm:factorization}. The number of interval
transforms is $2q$, the lifted intervals $(0,1)$ and $(0,c^{q})$ are real and
positive, and the parity of $p$ enters only through the evaluation point
$(-1)^{p}z^{q}$.
\end{theorem}
\begin{proof}
This is the componentwise reading of the factorization \eqref{eq:factor-thm} and
its holomorphic extension.
\end{proof}

The analytic lift of the integration path to the finite-sheeted plane is illustrated in Figure~\ref{fig:analytic_lift}, where the lifted continuous path successfully avoids the branch cut.

\begin{figure}[htbp]
    \centering
    \includegraphics[width=0.95\textwidth]{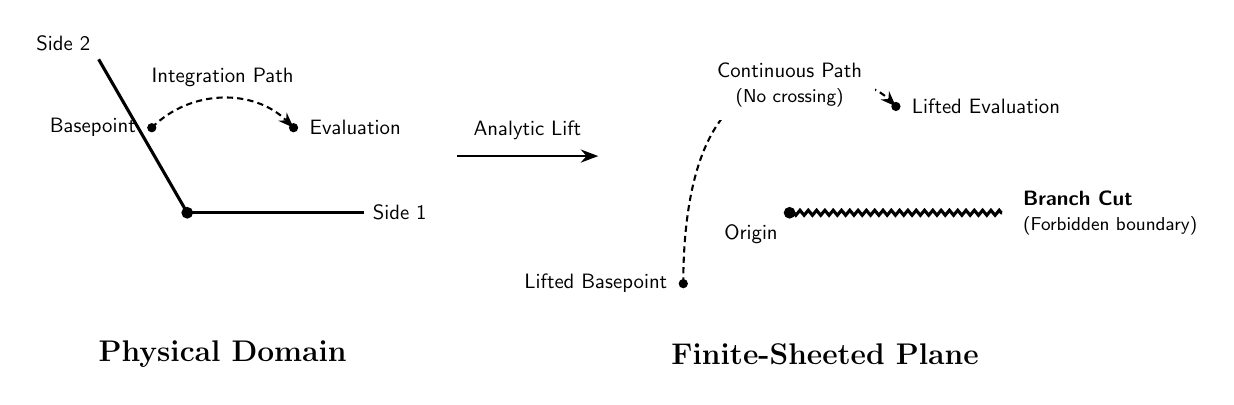}
    \caption{Analytic lift of an integration path from the physical rational wedge $\Gamma_{\theta,c}$ (left) to the finite-sheeted plane (right). Guided by the covering $w=\zeta^q$, the lift transports the basepoint and evaluation point while ensuring the continuous integration path strictly avoids the forbidden boundary of the branch cut.}
    \label{fig:analytic_lift}
\end{figure}

\section{Conormal endpoint spaces and lifted Plemelj theory}
\label{sec:conormal}

\subsection{Conormal weighted H\"older spaces}
The natural endpoint scale is conormal, controlled by the Euler field
$r\partial_r$ rather than $\partial_r$. This is forced by the lift: the
substitution $s=r^{q}$ does not preserve ordinary $C^{m_0,\beta}$ regularity at
the vertex, since $\partial_s[h(s^{1/q})]=\tfrac1q h'(s^{1/q})s^{1/q-1}$ blows
up for $q>1$, whereas $r\partial_r$ transforms cleanly
(Lemma~\ref{lem:transport}). Conormal scales are, in any case, the standard
framework for corner asymptotics.

\begin{definition}[Conormal weighted H\"older spaces]\label{def:weighted}
For $\alpha\in\C$, $m_0\in\N_0$, $0<\beta<1$, $a>0$, let
\[
  C^{m_0,\beta}_{\alpha}(0,a)
  =\bigl\{f(r)=r^{\alpha}h(r):\ (r\partial_r)^{k}h\in C^{\beta}[0,a],\
        0\le k\le m_0\bigr\},
  \quad
  \|f\|_{C^{m_0,\beta}_\alpha}
  =\sum_{k=0}^{m_0}\bigl\|(r\partial_r)^{k}(r^{-\alpha}f)\bigr\|_{C^{\beta}[0,a]} ,
\]
where $\|\phi\|_{C^{\beta}[0,a]}=\|\phi\|_{L^\infty}+[\phi]_{\beta}$ and
$[\phi]_\beta=\sup_{x\neq y}|\phi(x)-\phi(y)|/|x-y|^{\beta}$. The wedge space is
$C^{m_0,\beta}_{\bd\alpha}(\Gw)=C^{m_0,\beta}_{\alpha_1}(0,1)\oplus
C^{m_0,\beta}_{\alpha_2}(0,c)$, with norm
$\|F\|=\|f_1\|_{C^{m_0,\beta}_{\alpha_1}}+\|f_2\|_{C^{m_0,\beta}_{\alpha_2}}$,
acting on the side restrictions \eqref{eq:sides}.
\end{definition}

\begin{assumption}[Vertex integrability]\label{ass:integrable}
$\Real\alpha_j>-1$ for $j=1,2$.
\end{assumption}

This is exactly the condition that \eqref{eq:cauchy-def} converge absolutely at
the vertex, and it is the standing hypothesis below.

\subsection{The conormal mapping theorem}
We first record the elementary H\"older behaviour of the root substitution.

\begin{lemma}[Root composition]\label{lem:root-comp}
Let $\phi\in C^{\beta}[0,a]$ and $\psi(s)=s^{1/q}$ on $[0,a^{q}]$. Then
$\phi\circ\psi\in C^{\beta/q}[0,a^{q}]$ and
\[
  \|\phi\circ\psi\|_{C^{\beta/q}[0,a^{q}]}\le \|\phi\|_{C^{\beta}[0,a]} .
\]
\end{lemma}

\begin{proof}
The supremum is unchanged: $\|\phi\circ\psi\|_{L^\infty[0,a^{q}]}
=\|\phi\|_{L^\infty[0,a]}$. For the seminorm, $t\mapsto t^{1/q}$ is
$\tfrac1q$-H\"older with constant $1$ on $[0,\infty)$, i.e.
$|s_1^{1/q}-s_2^{1/q}|\le|s_1-s_2|^{1/q}$; hence
$|\phi(s_1^{1/q})-\phi(s_2^{1/q})|\le[\phi]_{\beta}|s_1^{1/q}-s_2^{1/q}|^{\beta}
\le[\phi]_{\beta}|s_1-s_2|^{\beta/q}$, so
$[\phi\circ\psi]_{\beta/q}\le[\phi]_{\beta}$. Summing the two bounds gives the
claim.
\end{proof}

\begin{lemma}[Conormal exponent transport]\label{lem:transport}
Let $f_j(r)=r^{\alpha_j}h_j(r)\in C^{m_0,\beta}_{\alpha_j}(0,a_j)$, $a_1=1$,
$a_2=c$, under Assumption~\ref{ass:integrable}. Then for $m=0,\dots,q-1$,
\begin{equation}\label{eq:transport}
  g_{j,m}(s)=s^{\beta_{j,m}}H_{j}(s),\qquad
  H_{j}(s)=h_j\!\bigl(s^{1/q}\bigr),\qquad
  \beta_{j,m}:=\frac{\alpha_j+m+1}{q}-1 ,
\end{equation}
with the exact conormal transformation rule
\begin{equation}\label{eq:conormal-rule}
  (s\partial_s)^{k}H_j(s)
  =q^{-k}\,\bigl[(r\partial_r)^{k}h_j\bigr]\!\bigl(s^{1/q}\bigr),
  \qquad 0\le k\le m_0 .
\end{equation}
Consequently $g_{j,m}\in C^{m_0,\beta/q}_{\beta_{j,m}}(0,a_j^{q})$ and
$\Real\beta_{j,m}>-1$ for all $m\ge0$.
\end{lemma}

\begin{proof}
Factorization \eqref{eq:transport} is immediate from \eqref{eq:gjm}. For
\eqref{eq:conormal-rule}, the case $k=1$ is the chain rule,
$s\partial_s H_j=\tfrac1q s^{1/q}h_j'(s^{1/q})=\tfrac1q[(r\partial_r)h_j](s^{1/q})$,
and induction applies the $k=1$ rule to $(r\partial_r)^{k}h_j$. Since
$(r\partial_r)^{k}h_j\in C^{\beta}[0,a_j]$, Lemma~\ref{lem:root-comp} gives
$(s\partial_s)^{k}H_j\in C^{\beta/q}[0,a_j^{q}]$, whence the membership. Finally
$\Real\beta_{j,m}=(\Real\alpha_j+m+1)/q-1>-1$ iff $\Real\alpha_j+m+1>0$, which
holds for all $m\ge0$ under Assumption~\ref{ass:integrable}.
\end{proof}

\begin{theorem}[Conormal mapping theorem]\label{thm:Lq-bounded}
Under Assumption~\ref{ass:integrable}, the lifting operator
\[
  \Lq:\ C^{m_0,\beta}_{\bd\alpha}(\Gw)
  \longrightarrow
  \bigoplus_{j=1}^{2}\bigoplus_{m=0}^{q-1}
     C^{m_0,\beta/q}_{\beta_{j,m}}(0,a_j^{q})
\]
is bounded, with the explicit estimate, independent of $\alpha,\beta,a,m_0$,
\begin{equation}\label{eq:Lq-bound}
  \|\Lq F\|
  =\sum_{j=1}^{2}\sum_{m=0}^{q-1}\|g_{j,m}\|_{C^{m_0,\beta/q}_{\beta_{j,m}}}
  \le q\,\|F\| .
\end{equation}
The sheet-count bound $q$ is sharp: $\|\Lq\|=q$.
More precisely, for each fixed $j$ all $q$ components have equal norm,
\begin{align}\label{eq:Lq-component}
\begin{aligned}
  \|g_{j,m}\|_{C^{m_0,\beta/q}_{\beta_{j,m}}}
  =\sum_{k=0}^{m_0}\bigl\|(s\partial_s)^{k}H_j\bigr\|_{C^{\beta/q}}
  =&\sum_{k=0}^{m_0}q^{-k}\bigl\|[(r\partial_r)^{k}h_j]\circ(\cdot)^{1/q}\bigr\|_{C^{\beta/q}} \\
  \le&\sum_{k=0}^{m_0}q^{-k}\bigl\|(r\partial_r)^{k}h_j\bigr\|_{C^{\beta}}
  \le\|f_j\|_{C^{m_0,\beta}_{\alpha_j}} .
\end{aligned}
\end{align}
\end{theorem}

\begin{proof}
By \eqref{eq:transport} the conormal-reduced function of $g_{j,m}$ is
$s^{-\beta_{j,m}}g_{j,m}=H_j$, independent of $m$; hence the first equality in
\eqref{eq:Lq-component}. The second uses \eqref{eq:conormal-rule}, and the third
is Lemma~\ref{lem:root-comp} applied to $\phi=(r\partial_r)^{k}h_j$. Since
$q^{-k}\le1$, the sum is at most $\sum_k\|(r\partial_r)^{k}h_j\|_{C^{\beta}}
=\|f_j\|_{C^{m_0,\beta}_{\alpha_j}}$. Summing \eqref{eq:Lq-component} over the
$q$ values of $m$ and over $j=1,2$ gives
$\|\Lq F\|\le q(\|f_1\|+\|f_2\|)=q\|F\|$. Sharpness: take $F=(f_1,0)$ with
$f_1(r)=r^{\alpha_1}$, so $h_1\equiv1$ and $\|F\|=\|h_1\|_{C^{\beta}}=1$. Then
$H_1\equiv1$, every conormal derivative $(s\partial_s)^{k}H_1$ vanishes for
$k\ge1$, and $\|g_{1,m}\|=\|H_1\|_{C^{\beta/q}}=1$ for each $m$; hence
$\|\Lq F\|=q=q\|F\|$, so $\|\Lq\|=q$.
\end{proof}

\begin{remark}
The operator norm is exactly $q$ and is independent of $\beta$, $a$, and
$\alpha$: the only cost of resolving the corner is the $q$-fold multiplicity of
sheets. The H\"older exponent degrades from $\beta$ to $\beta/q$, the
unavoidable price of the root substitution, while the conormal order $m_0$ is
preserved exactly.
\end{remark}

\subsection{Interval Cauchy mapping and boundary values}
After the lift one needs the mapping properties of the interval Cauchy operator.
We record a clean local version sufficient for the calculus; the strongest
singular-integral statements are not needed.

\begin{proposition}[Interval Cauchy mapping]\label{prop:interval-mapping}
Let $\mu\in\C$ with $\Real\mu>-1$, $a>0$, and $g\in C^{0,\beta}_{\mu}(0,a)$.
\begin{enumerate}
\item $\Cau_{(0,a)}[g]\in\Ahol(\C\setminus[0,a])$.
\item \textup{(Separated bound.)} For compact $K\subset\C$ with
$d=\dist(K,[0,a])>0$,
\[
  \bigl\|\Cau_{(0,a)}[g]\bigr\|_{L^\infty(K)}
  \le\frac{1}{2\pi d}\,\|g\|_{L^1(0,a)}
  \le\frac{1}{2\pi d}\,\frac{a^{\Real\mu+1}}{\Real\mu+1}\,
     \|g\|_{C^{0,\beta}_{\mu}} .
\]
\item \textup{(Endpoint asymptotics.)} The behaviour of $\Cau_{(0,a)}[g](\xi)$
as $\xi\to0$ off $[0,a]$ is given by the polyhomogeneous model expansion
developed in \S\ref{sec:propagation}, with leading term governed by the endpoint
exponent $\mu$; this is recorded here only as a pointer and is not used in the
proof of the present proposition.
\item \textup{(Plemelj boundary values.)} On every compact subinterval
$K\Subset(0,a)$ the non-tangential limits $\Cau_{(0,a)}^{\pm}[g]$ exist, are
H\"older-$\beta$, and satisfy
$\Cau_{(0,a)}^{+}[g]-\Cau_{(0,a)}^{-}[g]=g$ on $K$
\textup{\cite{muskhelishvili2008singular,gakhov2014boundary}}.
\end{enumerate}
\end{proposition}

\begin{proof}
(1) is standard differentiation under the integral for $\xi$ off $[0,a]$. For
(2), $|r-\xi|\ge d$ on $[0,a]$ gives
$|\Cau_{(0,a)}[g](\xi)|\le\frac{1}{2\pi}\int_0^a|g|/|r-\xi|\,dr
\le\frac{1}{2\pi d}\|g\|_{L^1}$; writing $g=r^{\mu}\tilde h$ with
$\|\tilde h\|_{L^\infty}\le\|g\|_{C^{0,\beta}_{\mu}}$,
$\|g\|_{L^1}\le\|\tilde h\|_{L^\infty}\int_0^a r^{\Real\mu}dr
=\|\tilde h\|_{L^\infty}a^{\Real\mu+1}/(\Real\mu+1)$, finite under
$\Real\mu>-1$. Item (3) is a pointer to \S\ref{sec:propagation}
(Lemma~\ref{lem:polyhom-model}) and requires nothing here. (4) is the classical
Plemelj--Privalov theorem applied on a compact subinterval, where
$g\in C^{\beta}$.
\end{proof}

\begin{corollary}[Mapping of the wedge transform]\label{cor:wedge-mapping}
Under Assumption~\ref{ass:integrable}, for $f\in C^{m_0,\beta}_{\bd\alpha}(\Gw)$
the transform $\Cau_{\Gw}f=\Eop\Cint\Lq f$ lies in $\Ahol(\C\setminus\Gw)$,
satisfies the separated bound on any compact $K$ with $\dist(K,\Gw)>0$ (with
constant controlled by $\|F\|$ through \eqref{eq:Lq-bound} and the polynomial
factors of $\Eop$), and has the boundary values and jump described in
Theorem~\ref{thm:plemelj}.
\end{corollary}

\begin{proof}
Each factor in \eqref{eq:factor-thm} is bounded on the relevant space:
$\Lq$ by Theorem~\ref{thm:Lq-bounded}, $\Cint$ by
Proposition~\ref{prop:interval-mapping}, and $\Eop$ by composition with the
polynomial weights $z^{q-1-m}$, $(e^{-i\theta}z)^{q-1-m}$ and the maps
$z\mapsto z^{q}$, $z\mapsto(-1)^{p}z^{q}$, all bounded on compact $K$. The
boundary behaviour is Theorem~\ref{thm:plemelj}.
\end{proof}

\subsection{Plemelj jump relations}

\begin{theorem}[Boundary values and jump relations]\label{thm:plemelj}
Let $f\in C^{m_0,\beta}_{\bd\alpha}(\Gw)$ satisfy
Assumption~\ref{ass:integrable}. Then $\Cau_{\Gw}f$ has non-tangential boundary
values $\Cau_{\Gw}^{\pm}f$ from either side of each open edge, continuous up to
the open edges, and away from the vertex
\begin{equation}\label{eq:jump}
  \Cau_{\Gw}^{+}f(\zeta)-\Cau_{\Gw}^{-}f(\zeta)=f(\zeta),
  \qquad \zeta\in\Gw\setminus\{0,1,ce^{i\theta}\}.
\end{equation}
The vertex asymptotics of $\Cau_{\Gw}^{\pm}f$ follow from
Theorem~\ref{thm:propagation} by passing to boundary values of the sectorwise
powers.
\end{theorem}

\begin{proof}
Use the rotation reduction \eqref{eq:reduction}. On a compact subinterval of the
first side, the second term $\Cau_{(0,c)}[f_2](e^{-i\theta}z)$ is holomorphic
(its argument stays off $(0,c)$), and
Proposition~\ref{prop:interval-mapping}(4) gives
$\Cau_{(0,1)}^{+}[f_1]-\Cau_{(0,1)}^{-}[f_1]=f_1=f$ there. On a compact
subinterval of the second side, the first term is holomorphic and
$e^{-i\theta}z$ crosses $(0,c)$ transversally, giving the jump $f_2=f$. This is
\eqref{eq:jump}; the boundary values exist and are H\"older by the same interval
theorem. The finite-sheeted form \eqref{eq:main} is not needed for the
open-edge jump; it governs the endpoint behaviour, treated in
\S\ref{sec:propagation}.
\end{proof}

\section{Mellin asymptotics and polyhomogeneous corner-mode propagation}
\label{sec:propagation}

\subsection{Sectors and branch convention}
The polyhomogeneous expansions established here are the transform-side
counterpart of the classical corner asymptotics for elliptic problems
\cite{grisvard2011elliptic,kozlov1997elliptic,dauge2006elliptic,costabel2003asymptotics,schulze1990mellin}: the densities carry expansions in
powers and logarithms of the distance to the vertex, and we determine how the
Cauchy operator propagates them. The decomposition \eqref{eq:main} is
branch-free, but the vertex expansion below involves non-integer powers and
logarithms, which require a branch. The two rays of $\Gw$ divide a punctured
neighbourhood of $0$ into the open sectors
\[
  \Sigma_{\mathrm I}=\{\,0<\arg z<\theta\,\},\qquad
  \Sigma_{\mathrm{II}}=\{\,\theta<\arg z<2\pi\,\}.
\]
On each sector we use the determination of $\arg z$ from these ranges and set
$z^{\alpha}=|z|^{\alpha}e^{i\alpha\arg z}$, $\log z=\ln|z|+i\arg z$. For the
model interval transform we use the principal cut: $(-\xi)^{\alpha}
=\exp(\alpha\Log(-\xi))$, where $\Log$ is the principal logarithm on
$\C\setminus(-\infty,0]$, so $(-\xi)^{\alpha}$ is cut along $\xi\in(0,\infty)$
and positive for $\xi<0$. All expansions are understood separately on
$\Sigma_{\mathrm I}$ and $\Sigma_{\mathrm{II}}$, with sector-dependent
coefficients; we write a generic sector label $\bullet\in\{\mathrm I,\mathrm{II}\}$.

\subsection{The Mellin model: pure powers}

\begin{lemma}[Model interval expansion]\label{lem:model}
Let $\alpha\in\C$, $\Real\alpha>-1$, $\alpha\notin\Z_{\ge0}$, $a>0$. Then, as
$\xi\to0$ off $[0,a]$,
\begin{equation}\label{eq:model}
  \Cau_{(0,a)}[s^{\alpha}](\xi)
  =m(\alpha)\,(-\xi)^{\alpha}
   +\frac{1}{2\pi i}\sum_{n\ge0}\frac{a^{\alpha-n}}{\alpha-n}\,\xi^{n},
   \qquad m(\alpha)=\frac{1}{2i\sin\pi(\alpha+1)},
\end{equation}
as an asymptotic expansion. For fixed $\xi\notin[0,a]$ the left side is
holomorphic in $\alpha$ on $\{\Real\alpha>-1\}$; the two singular contributions
at any $\alpha\to n_0\in\Z_{\ge0}$ have canceling simple poles, with combined
residue zero, and the value at $\alpha=n_0$ contains the term
$-\tfrac{1}{2\pi i}\xi^{n_0}\Log(-\xi)$.
\end{lemma}

\begin{proof}
Write $\Cau_{(0,a)}[s^{\alpha}](\xi)=\frac{1}{2\pi i}(I_\infty-I_a)$ with
$I_\infty=\int_0^\infty s^{\alpha}(s-\xi)^{-1}ds$,
$I_a=\int_a^\infty s^{\alpha}(s-\xi)^{-1}ds$. For $-1<\Real\alpha<0$, scaling
$\int_0^\infty x^{\sigma-1}(x+1)^{-1}dx=\pi/\sin\pi\sigma$ ($0<\Real\sigma<1$)
by $x=s/(-\xi)$, $\sigma=\alpha+1$, gives
$I_\infty=\pi(-\xi)^{\alpha}/\sin\pi(\alpha+1)$. The uniform expansion
$(s-\xi)^{-1}=\sum_{n\ge0}\xi^{n}s^{-n-1}$ on $s\ge a$ integrates termwise to
$I_a=-\sum_{n}\xi^{n}a^{\alpha-n}/(\alpha-n)$. This is \eqref{eq:model} on
$-1<\Real\alpha<0$; both sides are holomorphic in $\alpha$ on
$\{\Real\alpha>-1\}\setminus\Z_{\ge0}$, so the identity continues in $\alpha$.
For the residue claim, the residue of $\frac{1}{2i\sin\pi(\alpha+1)}$ at
$\alpha=n_0$ is $\frac{(-1)^{n_0+1}}{2\pi i}$, so
$\mathrm{Res}_{\alpha=n_0}[m(\alpha)(-\xi)^{\alpha}]
=\frac{(-1)^{n_0+1}}{2\pi i}(-\xi)^{n_0}=-\tfrac{1}{2\pi i}\xi^{n_0}$, exactly
cancelling the residue $\tfrac{1}{2\pi i}\xi^{n_0}$ of the $n=n_0$ term. The
finite part is computed in Lemma~\ref{lem:polyhom-model}.
\end{proof}

\subsection{The Mellin model: powers times logarithms}
The key device is differentiation in the exponent,
\begin{equation}\label{eq:dalpha}
  s^{\alpha}(\log s)^{h}=\partial_\alpha^{h}s^{\alpha},
  \qquad
  (-\xi)^{\alpha}\bigl(\Log(-\xi)\bigr)^{h}=\partial_\alpha^{h}(-\xi)^{\alpha}.
\end{equation}

\begin{lemma}[Polyhomogeneous model expansion]\label{lem:polyhom-model}
Let $\alpha\in\C$, $\Real\alpha>-1$, $h\in\N_0$, $a>0$. For fixed
$\xi\notin[0,a]$,
\begin{equation}\label{eq:dalpha-transform}
  \Cau_{(0,a)}\bigl[s^{\alpha}(\log s)^{h}\bigr](\xi)
  =\partial_\alpha^{h}\,\Cau_{(0,a)}[s^{\alpha}](\xi),
\end{equation}
and the asymptotic expansion of \eqref{eq:model} may be differentiated termwise
in $\alpha$. Consequently:
\begin{enumerate}
\item \textup{(Nonresonant, $\alpha\notin\Z_{\ge0}$.)} As $\xi\to0$,
\begin{equation}\label{eq:polyhom-nonres}
  \Cau_{(0,a)}\bigl[s^{\alpha}(\log s)^{h}\bigr](\xi)
  =\sum_{i=0}^{h}\binom{h}{i}m^{(i)}(\alpha)\,(-\xi)^{\alpha}
     \bigl(\Log(-\xi)\bigr)^{h-i}
   +\sum_{n\ge0}c_{n,h}\,\xi^{n},
\end{equation}
with constant coefficients $c_{n,h}$ depending on $a,\alpha,n$ (no logarithm
occurs in the regular part when $\alpha\notin\Z_{\ge0}$). The top logarithmic
power of the singular part is $h$, with coefficient $m(\alpha)$.
\item \textup{(Resonant, $\alpha=n_0\in\Z_{\ge0}$.)} As $\xi\to0$,
\begin{equation}\label{eq:polyhom-res}
  \Cau_{(0,a)}\bigl[s^{n_0}(\log s)^{h}\bigr](\xi)
  =\frac{-1}{2\pi i\,(h+1)}\,\xi^{n_0}\bigl(\Log(-\xi)\bigr)^{h+1}
   +\sum_{0\le h'\le h}c_{h'}\,\xi^{n_0}\bigl(\Log(-\xi)\bigr)^{h'}
   +\sum_{n\ge0}\widetilde P_{n,h}(\log\xi)\,\xi^{n},
\end{equation}
with explicit constants $c_{h'}$ and polynomials $\widetilde P_{n,h}$ of degree
$\le h+1$. The top logarithmic power is $h+1$, raised by one relative to the
density.
\end{enumerate}
\end{lemma}

\begin{proof}
For fixed $\xi\notin[0,a]$ the integrand of
$\Cau_{(0,a)}[s^{\alpha}](\xi)=\frac{1}{2\pi i}\int_0^a s^{\alpha}(s-\xi)^{-1}ds$
is, by \eqref{eq:dalpha}, differentiated in $\alpha$ to give
$s^{\alpha}(\log s)^{h}(s-\xi)^{-1}$, which is absolutely integrable on $(0,a)$
for $\Real\alpha>-1$ and locally uniformly in $\alpha$; hence
\eqref{eq:dalpha-transform}. The expansion \eqref{eq:model} is the sum over poles
of the Mellin transform of $(s-\xi)^{-1}$ and is holomorphic in $\alpha$ on
$\{\Real\alpha>-1\}\setminus\Z_{\ge0}$, so its terms may be differentiated in
$\alpha$, yielding the asymptotic expansion of the left side of
\eqref{eq:dalpha-transform}.

\emph{Nonresonant.} Apply $\partial_\alpha^{h}$ to \eqref{eq:model}. By the
Leibniz rule and the second identity of \eqref{eq:dalpha},
$\partial_\alpha^{h}[m(\alpha)(-\xi)^{\alpha}]
=\sum_{i=0}^{h}\binom{h}{i}m^{(i)}(\alpha)(-\xi)^{\alpha}(\Log(-\xi))^{h-i}$,
the singular part of \eqref{eq:polyhom-nonres}. 
For the regular part,
$\partial_\alpha^{h}[a^{\alpha-n}(\alpha-n)^{-1}\xi^{n}]$ is, for
$\alpha\notin\Z_{\ge0}$, a constant (a polynomial of degree $\le h$ in $\ln a$)
times $\xi^{n}$, giving $c_{n,h}\xi^{n}$; no $\log\xi$ occurs in the regular
part for nonresonant $\alpha$.

\emph{Resonant.} Fix $n_0\in\Z_{\ge0}$ and set $\eta=\alpha-n_0$. Near $\eta=0$,
$m(\alpha)=\frac{\mu_{-1}}{\eta}+\mu_1\eta+\cdots$ (odd Laurent expansion, since
$1/\sin$ is odd about its pole) with $\mu_{-1}=\frac{(-1)^{n_0+1}}{2\pi i}$, and
$(-\xi)^{\alpha}=(-1)^{n_0}\xi^{n_0}e^{\eta L}$, $L:=\Log(-\xi)$. The $n=n_0$
regular term is $\frac{\xi^{n_0}}{2\pi i}\,\eta^{-1}e^{\eta\ln a}$. Their sum,
the only part singular as $\eta\to0$, is
\[
  G(\alpha,\xi)
  =\mu_{-1}(-1)^{n_0}\xi^{n_0}\,\frac{e^{\eta L}-e^{\eta\ln a}}{\eta}
   +\bigl(\text{terms regular in }\eta\bigr),
\]
using $\frac{\xi^{n_0}}{2\pi i}=-\mu_{-1}(-1)^{n_0}\xi^{n_0}$ to combine the two
$\eta^{-1}$ residues (which therefore cancel, confirming
Lemma~\ref{lem:model}). 
Now
$\frac{e^{\eta L}-e^{\eta\ln a}}{\eta}=\sum_{q\ge0}\frac{\eta^{q}}{(q+1)!}
\bigl(L^{q+1}-(\ln a)^{q+1}\bigr)$, so
$\partial_\alpha^{h}G|_{\alpha=n_0}=h!\,[\eta^{h}]G$ contributes
$\mu_{-1}(-1)^{n_0}\xi^{n_0}\frac{1}{h+1}\bigl(L^{h+1}-(\ln a)^{h+1}\bigr)$. The
non-analytic leading term is
$\mu_{-1}(-1)^{n_0}\frac{1}{h+1}\xi^{n_0}L^{h+1}
=\frac{-1}{2\pi i(h+1)}\xi^{n_0}L^{h+1}$, since
$\mu_{-1}(-1)^{n_0}=\frac{(-1)^{2n_0+1}}{2\pi i}=\frac{-1}{2\pi i}$. The
regular-in-$\eta$ remainder of $G$ together with the other $n\neq n_0$ regular
terms supplies the lower powers $L^{h'}$, $h'\le h$, and the analytic series with
log-coefficients of degree $\le h+1$; this is \eqref{eq:polyhom-res}.
\end{proof}

\begin{remark}
The mechanism is transparent: away from integers $m(\alpha)$ is regular and the
density's log power passes through unchanged; at an integer exponent the simple
pole of $m(\alpha)$ collides with the integer mode of the regular series, and
each differentiation in $\alpha$ that produces the resonant log adds one power,
giving $h\mapsto h+1$.
\end{remark}

\subsection{Polyhomogeneous propagation}

\begin{theorem}[Polyhomogeneous corner-mode propagation]\label{thm:propagation}
Let $f\in C^{m_0,\beta}_{\bd\alpha}(\Gw)$ satisfy
Assumption~\ref{ass:integrable}, and suppose the side densities have classical
conormal (polyhomogeneous) expansions
\[
  f_j(r)\sim\sum_{\ell\ge0}\sum_{h=0}^{N_\ell}
     a_{j,\ell,h}\,r^{\alpha_{j,\ell}}(\log r)^{h}
  \qquad(r\downarrow0),\qquad \Real\alpha_{j,\ell}\uparrow\infty,\ j=1,2 .
\]
Then on each sector $\bullet\in\{\mathrm I,\mathrm{II}\}$, as $z\to0$,
\begin{equation}\label{eq:propagation}
  \Cau_{\Gw}f(z)\sim
  \sum_{j=1}^{2}\sum_{\ell\ge0}\sum_{h'=0}^{N_\ell+\delta_{j,\ell}}
     A^{\bullet}_{j,\ell,h'}\,z^{\alpha_{j,\ell}}(\log z)^{h'}
  +\sum_{\nu\ge0}\sum_{h'}B^{\bullet}_{\nu,h'}\,z^{\nu}(\log z)^{h'},
\end{equation}
where $\delta_{j,\ell}=1$ if $\alpha_{j,\ell}\in\Z_{\ge0}$ (resonant) and $0$
otherwise. The coefficients are explicit: writing $\varphi_1(z)=z$,
$\varphi_2(z)=e^{-i\theta}z$, each density mode $a_{j,\ell,h}r^{\alpha}(\log r)^h$
contributes, via $\Cau_{(0,a_j)}[\cdot](\varphi_j(z))$ and
Lemma~\ref{lem:polyhom-model}, the singular family
$\sum_i\binom{h}{i}m^{(i)}(\alpha)(-\varphi_j(z))^{\alpha}(\Log(-\varphi_j(z)))^{h-i}$
in the nonresonant case and the family \eqref{eq:polyhom-res} with
$\xi=\varphi_j(z)$ in the resonant case. Re-expanding $(-\varphi_j(z))^{\alpha}
=z^{\alpha}\cdot\kappa^{\bullet}_{j,\alpha}$ and
$\Log(-\varphi_j(z))=\log z+\lambda^{\bullet}_{j}$ with sector-constant
$\kappa^{\bullet}_{j,\alpha},\lambda^{\bullet}_{j}$ (the entire $(p,q)$
dependence entering through $\theta=p\pi/q$ in $\varphi_2$) collects the
coefficients $A^{\bullet}_{j,\ell,h'}$. The analytic coefficients
$B^{\bullet}_{\nu,h'}$ come from the regular parts of $f_1,f_2$ and the far
endpoints $1,ce^{i\theta}$.

In particular, a nonresonant mode preserves the top logarithmic power $h$; a
resonant mode $\alpha_{j,\ell}=n_0\in\Z_{\ge0}$ with top density log power
$h$ contributes the top output power $h+1$, with coefficient
\[
  a_{j,\ell,h}\,\frac{-1}{2\pi i\,(h+1)}\,\kappa^{\bullet}_{j,\alpha_{j,\ell}} .
\]
Here each $A^{\bullet}_{j,\ell,h'}$ denotes the \emph{collected} coefficient of
$z^{\alpha_{j,\ell}}(\log z)^{h'}$, i.e. the sum over the density modes
$(j,\ell,h)$ of their individual contributions; the displayed formula is the
single-mode contribution to the highest power $h'=h+1$, which receives no other
mode.
\end{theorem}

\begin{proof}
By Proposition~\ref{prop:reduction},
$\Cau_{\Gw}f=\Cau_{(0,1)}[f_1](z)+\Cau_{(0,c)}[f_2](e^{-i\theta}z)$. The far
endpoints $r=1,c$ lie at positive distance from $z\to0$, so contribute only the
analytic family $\sum_{\nu,h'}B^{\bullet}_{\nu,h'}z^{\nu}(\log z)^{h'}$ (the
$\log$ powers arising from the regular parts of \eqref{eq:polyhom-nonres} and
\eqref{eq:polyhom-res}). Insert the conormal expansions of $f_j$ and apply
Lemma~\ref{lem:polyhom-model} termwise to each mode, with $\xi=\varphi_j(z)$;
this is justified by the classical asymptotic calculus for polyhomogeneous
densities, the truncation error of order $N$ propagating to order $N$ by the
separated bound of Proposition~\ref{prop:interval-mapping}(2). Side $1$ uses
$\varphi_1(z)=z$ directly. Side $2$ uses $\varphi_2(z)=e^{-i\theta}z$, so on each
sector $(-e^{-i\theta}z)^{\alpha}=z^{\alpha}\kappa^{\bullet}_{2,\alpha}$ and
$\Log(-e^{-i\theta}z)=\log z+\lambda^{\bullet}_{2}$ with sector constants; the
resulting binomial re-expansion of $(\log z+\lambda^{\bullet}_{2})^{h-i}$
redistributes among the powers $(\log z)^{h'}$, $h'\le h$ (or $\le h+1$ at a
resonance). Collecting terms gives \eqref{eq:propagation}; the leading resonant
coefficient is read from \eqref{eq:polyhom-res}.
\end{proof}

Figure~\ref{fig:asymptotic_propagation} illustrates the propagation rules for the asymptotic modes, highlighting the explicit logarithmic power increase induced by the resonance condition.

\begin{figure}[htbp]
    \centering
    \includegraphics[width=0.85\textwidth]{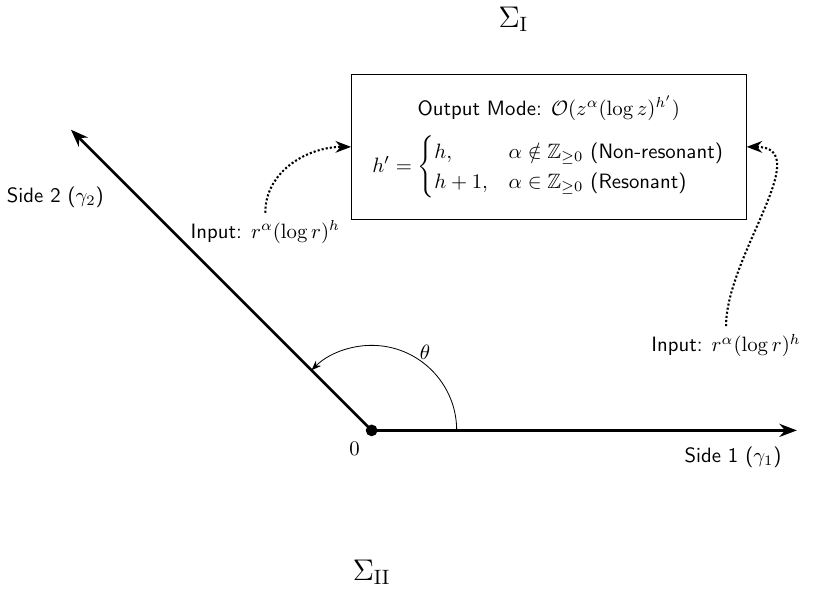}
    \caption{Propagation mechanism of the logarithmic asymptotic modes near the vertex of the wedge. The logarithmic power strictly increases by one under the resonance condition ($\alpha \in \mathbb{Z}_{\ge 0}$).}
    \label{fig:asymptotic_propagation}
\end{figure}

\begin{remark}[Finite-sheeted recombination]
The same expansion follows from the factorization \eqref{eq:factor-thm} and
exhibits its mechanism. By Lemma~\ref{lem:transport} the lifted mode in sheet
$(j,m)$ has exponent $\beta_{j,m}^{(\ell)}=(\alpha_{j,\ell}+m+1)/q-1$, and the
$m$-th block contributes $z^{q-1-m}(z^{q})^{\beta_{j,m}^{(\ell)}}=z^{\alpha_{j,\ell}}$,
independently of $m$, since the exponents add:
$(q-1-m)+(\alpha_{j,\ell}+m+1-q)=\alpha_{j,\ell}$. Thus, the single physical mode
$z^{\alpha_{j,\ell}}$ is split by the lift into $q$ interval endpoint modes of
common exponent but distinct sheet weights, whose recombination against the
branch factors and per-sheet Mellin constants reconstitutes
$A^{\bullet}_{j,\ell,\bullet}$. The covering $w=z^{q}$ uniformizes the angular
monodromy of $z^{\alpha}$ across the $q$ sectors; this organization of the
corner modes, not any reduction of the transform, is the role of the
finite-sheeted structure.
\end{remark}

\subsection{Relation to the Kerzman--Stein pathology}
For continuously differentiable curves, the Kerzman--Stein operator
$A=\Cau-\Cau^{*}$ is compact \cite{kerzman1978cauchy}; for piecewise
continuously differentiable curves it is not, and the finite symmetric wedge
furnishes an explicit essential spectrum \cite{bolt2015kerzman}. Notably,
\cite{bolt2015kerzman} represents $L^{2}$ of the symmetric wedge by pairs of functions
on a unit interval and, after a change of variable, realizes the operator as a
convolution on the line---the same passage from a two-sided wedge to interval
data that underlies the rotation reduction \eqref{eq:reduction}. The present
calculus is the constructive counterpart of that spectral diagnosis: the
factorization \eqref{eq:factor-thm} and the propagation
Theorem~\ref{thm:propagation} exhibit the algebraic modes
$z^{\alpha_{j,\ell}}(\log z)^{h'}$ and their branch coefficients
$A^{\bullet}_{j,\ell,h'}$ responsible for the loss of compactness, uniformized
by $w=z^{q}$.

\section{Logarithmic transforms and real-potential corrections}
\label{sec:log}

\subsection{The antiderivative relation}
Fix a branch of $\log(\zeta-z)$ on $\C\setminus\Gw$ and define
\begin{equation}\label{eq:log-defs}
  \Lop_{\Gw}f(z)=\int_{\Gw}\log|\zeta-z|\,f(\zeta)\,|d\zeta|,
  \qquad
  \Lcx_{\Gw}f(z)=\int_{\Gw}\log(\zeta-z)\,f(\zeta)\,d\zeta .
\end{equation}
Differentiation under the integral gives the antiderivative relation
\begin{equation}\label{eq:antideriv}
  \partial_z\Lcx_{\Gw}f(z)=-2\pi i\,\Cau_{\Gw}f(z),
  \qquad z\in\C\setminus\Gw,
\end{equation}
so the factorization \eqref{eq:factor-thm} determines $\Lcx_{\Gw}f$ up to a
holomorphic normalization. We characterize the logarithmic transform
accordingly rather than forcing a closed form.

\subsection{Finite-sheeted logarithmic decomposition}

\begin{theorem}[Logarithmic decomposition]\label{thm:log}
Let $f\in C^{m_0,\beta}_{\bd\alpha}(\Gw)$ satisfy
Assumption~\ref{ass:integrable}, fix a connected component $\Omega$ of
$\C\setminus\Gw$, a basepoint $z_0\in\Omega$, and compatible logarithm branches
on $\Omega$. Then on $\Omega$
\begin{equation}\label{eq:log-decomp}
  \Lcx_{\Gw}f(z)
  =\Lcx_{(0,1)}[f_1](z)
   +e^{i\theta}\,\Lcx_{(0,c)}[f_2]\!\bigl(e^{-i\theta}z\bigr)
   +\kappa_\Omega\,e^{i\theta}\!\!\int_0^c f_2(r)\,dr
   +K_{z_0},
\end{equation}
where $\kappa_\Omega=i\theta+2\pi iN_\Omega$ ($N_\Omega\in\Z$) records the
branch of the second-side logarithm on $\Omega$ and $K_{z_0}$ fixes
$\Lcx_{\Gw}f(z_0)$. The right side is characterized by \eqref{eq:antideriv}
together with this normalization. Applying the single-interval logarithmic lift
(Proposition~\ref{prop:log-lift}) to the two interval transforms yields an
antiderivative-level $2q$-term finite-sheeted representation in $z^{q}$ and
$(-1)^{p}z^{q}$, the algebraic and logarithmic corrections appearing as the
primitives and integration constants of that representation, rather than in
closed form.
\end{theorem}

\begin{proof}
The first side gives $\Lcx_{(0,1)}[f_1](z)$ directly. On the second side,
$\zeta=re^{i\theta}$, $d\zeta=e^{i\theta}dr$, and for the fixed branch on
$\Omega$, $\log(re^{i\theta}-z)=\kappa_\Omega+\log(r-e^{-i\theta}z)$ with
$\kappa_\Omega=i\theta+2\pi iN_\Omega$ constant on $\Omega$; integrating gives
the second and third terms of \eqref{eq:log-decomp}. Differentiating in $z$ and
using $\partial_z\Lcx_{(0,a)}[g](\xi)=-2\pi i\,\Cau_{(0,a)}[g](\xi)$ reproduces
$-2\pi i$ times the reduction \eqref{eq:reduction}; the $z$-independent terms
drop and are pinned by the value at $z_0$. The lifted form is
Proposition~\ref{prop:log-lift}.
\end{proof}

\begin{proposition}[Single-interval logarithmic lift]\label{prop:log-lift}
For $a>0$, $g\in L^1(0,a)$, and $\xi$ with $\xi^{q}\notin(0,a^{q})$,
\begin{equation}\label{eq:log-lift}
  \Lcx_{(0,a)}[g](\xi)
  =\frac{-2\pi i}{q}\sum_{m=0}^{q-1}\int_{\xi_0}^{\xi}\eta^{q-1-m}\,
     \Cau_{(0,a^{q})}[\tilde g_m](\eta^{q})\,d\eta
   +\Lcx_{(0,a)}[g](\xi_0),
\end{equation}
with $\tilde g_m$ as in Lemma~\ref{lem:lift} and the path in $\C\setminus[0,a]$.
Equivalently, the lifted logarithmic formula is the $\xi$-antiderivative of the
lifted Cauchy formula \eqref{eq:lift}, normalized at $\xi_0$.
\end{proposition}

\begin{proof}
From $\partial_\xi\Lcx_{(0,a)}[g]=-2\pi i\,\Cau_{(0,a)}[g]$ and
Lemma~\ref{lem:lift}, $\partial_\xi\Lcx_{(0,a)}[g](\xi)
=\frac{-2\pi i}{q}\sum_m\xi^{q-1-m}\Cau_{(0,a^{q})}[\tilde g_m](\xi^{q})$;
integrate along a path in $\C\setminus[0,a]$ and add the value at $\xi_0$. Thus
differentiating the lifted logarithmic formula returns precisely \eqref{eq:lift};
the logarithmic lift is its antiderivative, not a verbatim application of
Lemma~\ref{lem:lift}.
\end{proof}

\subsection{Real logarithmic transform}
The relation between $\Lop_{\Gw}$ and $\Lcx_{\Gw}$ is \emph{not} a global real
part: $\log|\zeta-z|=\Real\log(\zeta-z)$ holds pointwise, but the orientation
factor $d\zeta=e^{i\theta_s}|d\zeta|$ on a tilted side rotates the integrand
before the real part is taken.

\begin{proposition}[Corrected real-part relation]\label{prop:realpart}
Let side $s\in\{1,2\}$ have tangent phase $\theta_1=0$, $\theta_2=\theta$, and
write $\Lcx_s,\Lop_s$ for the single-side transforms.
\begin{enumerate}
\item If $f$ is real on side $s$, then
  $\Lop_s f(z)=\Real\!\bigl(e^{-i\theta_s}\Lcx_s f(z)\bigr)$, and
  $\Lop_{\Gw}f=\sum_s\Real\!\bigl(e^{-i\theta_s}\Lcx_s f\bigr)$.
\item For complex $f$,
  $\Lop_s f(z)=\tfrac12 e^{-i\theta_s}\Lcx_s f(z)
   +\tfrac12 e^{\,i\theta_s}\overline{\Lcx_s\bar f(z)}.$
\end{enumerate}
Here the bar is the pointwise conjugation of the resulting scalar function at the
same $z$ (not $\bar z$), using the same logarithm branch before conjugating. In
particular $\Lop_{\Gw}f\neq\Real(\Lcx_1 f+\Lcx_2 f)$ in general.
\end{proposition}

\begin{proof}
On side $s$, $|d\zeta|=e^{-i\theta_s}d\zeta$, so
$\Lop_s f=e^{-i\theta_s}\int_s\log|\zeta-z|f\,d\zeta$. With
$\log|\zeta-z|=\tfrac12(\log(\zeta-z)+\overline{\log(\zeta-z)})$, the first half
is $\tfrac12 e^{-i\theta_s}\Lcx_s f$. For the second, using
$\overline{d\zeta}=e^{-2i\theta_s}d\zeta$,
$\int_s\overline{\log(\zeta-z)}f\,d\zeta
=\overline{\int_s\log(\zeta-z)\bar f\,\overline{d\zeta}}
=e^{2i\theta_s}\overline{\Lcx_s\bar f(z)}$; multiply by $\tfrac12 e^{-i\theta_s}$
to obtain (2). For real $f$ the two terms are conjugate and sum to
$\Real(e^{-i\theta_s}\Lcx_s f)$, which is (1).
\end{proof}

\begin{remark}
The vertex asymptotics of $\Lop_{\Gw}f$ follow by integrating those of
$\Cau_{\Gw}f$ (Theorem~\ref{thm:propagation}) through \eqref{eq:antideriv} and
applying Proposition~\ref{prop:realpart}: a mode $z^{\alpha}(\log z)^{h}$
integrates to a combination of $z^{\alpha+1}(\log z)^{h'}$, $h'\le h$, and a
resonant $z^{n}(\log z)^{h}$ to a term in $z^{n+1}(\log z)^{h+1}$.
\end{remark}

\section{Rational-corner singular structure of Helmholtz layer potentials}
\label{sec:helmholtz}

\subsection{The local Helmholtz kernel}
For $\Delta u+k^{2}u=0$ in $\R^{2}$ the outgoing fundamental solution is
$\Phi_k(x,y)=\tfrac{i}{4}H_0^{(1)}(k|x-y|)$. Near the diagonal it has the local
structure
\begin{equation}\label{eq:phi-local}
  \Phi_k(x,y)
  =-\frac{1}{2\pi}\log|x-y|\,A_k\!\bigl(|x-y|^{2}\bigr)+B_k\!\bigl(|x-y|^{2}\bigr),
  \qquad A_k(0)=1,
\end{equation}
with $A_k,B_k$ analytic near $0$; indeed
$A_k(r^{2})=J_0(kr)$ and $B_k$ collects the remaining even analytic part. Thus
the leading singularity of $\Phi_k$ is the logarithmic kernel of
\S\ref{sec:log}; the term $(A_k-1)\log r$ is $O(r^{2}\log r)$, hence two orders
smoother at the diagonal than the leading singularity, while $B_k$ is analytic
at $r=0$.

\subsection{Model potential operators and the singular-operator algebra}
On the straight wedge $\Gw$, define the single-, double-, and adjoint
double-layer model operators
\begin{align}
  S_0 f(z)&=\int_{\Gw}\log|\zeta-z|\,f(\zeta)\,|d\zeta|=\Lop_{\Gw}f(z),\\
  K_0 f(z)&=\operatorname{p.v.}\!\int_{\Gw}
     \partial_{n(\zeta)}\log|\zeta-z|\,f(\zeta)\,|d\zeta|,\\
  K_0' f(z)&=\operatorname{p.v.}\!\int_{\Gw}
     \partial_{n(z)}\log|\zeta-z|\,f(\zeta)\,|d\zeta|,
\end{align}
with $n$ a choice of unit normal, constant on each side. Identifying $\R^2$ with
$\C$, write $\nu(\zeta)=n_1(\zeta)+in_2(\zeta)$ for the complex form of the
normal, so $\nu(\zeta)=\nu_s$ is constant on side $s$.

\begin{proposition}[Potential operators as Cauchy boundary operators]
\label{prop:potential-algebra}
With $\partial_\zeta=\tfrac12(\partial_{x}-i\partial_{y})$ one has, for
$\zeta\neq z$,
\begin{equation}\label{eq:normal-kernels}
  \partial_{n(\zeta)}\log|\zeta-z|=\Real\frac{\nu(\zeta)}{\zeta-z},
  \qquad
  \partial_{n(z)}\log|\zeta-z|=-\Real\frac{\nu(z)}{\zeta-z}.
\end{equation}
Consequently $S_0=\Lop_{\Gw}$, and $K_0,K_0'$ are real-linear combinations, with
side-constant coefficients $\nu_s$, of boundary values of the finite-sheeted
wedge Cauchy transform $\Cau_{\Gw}$ and its conjugate. Explicitly, for the
boundary value from a fixed side,
\begin{equation}\label{eq:K-as-cauchy}
  K_0 f=\Real\Bigl[\,2\pi i\sum_{s}\nu_s e^{-i\theta_s}\,\Cau^{\mathrm{bv}}_{s}f\Bigr],
  \qquad
  K_0' f(z)=-\Real\Bigl[\,2\pi i\,\nu(z)\sum_{s} e^{-i\theta_s}\,\Cau^{\mathrm{bv}}_{s}f(z)\Bigr]
  \ \ (\text{for real }f),
\end{equation}
where $\Cau^{\mathrm{bv}}_{s}$ denotes the boundary value of the side-$s$ Cauchy
transform; the factor $2\pi i$ is $+2\pi i$ because $\int_s f(\zeta)(\zeta-z)^{-1}
d\zeta=2\pi i\,\Cau_s f$ for the convention \eqref{eq:cauchy-def}. For complex
$f$ the conjugate term $\overline{\Cau_{\Gw}\bar f}$ appears, via the real-part
splitting of \eqref{eq:normal-kernels}.
\end{proposition}

\begin{proof}
For a real function $u$ and unit direction $\nu=n_1+in_2$,
$\partial_n u=2\Real(\nu\,\partial_\zeta u)$ (write
$\partial_x=\partial_\zeta+\partial_{\bar\zeta}$,
$\partial_y=i(\partial_\zeta-\partial_{\bar\zeta})$ and use
$\partial_{\bar\zeta}u=\overline{\partial_\zeta u}$). With
$u(\zeta)=\log|\zeta-z|=\tfrac12\log((\zeta-z)\overline{(\zeta-z)})$,
$\partial_\zeta u=\tfrac12(\zeta-z)^{-1}$, giving the first identity in
\eqref{eq:normal-kernels}; the second follows from
$\partial_z\log|\zeta-z|=-\tfrac12(\zeta-z)^{-1}$ and the same formula in $z$.
For real $f$, $K_0 f=\Real\int_{\Gw}\frac{\nu(\zeta)}{\zeta-z}f\,|d\zeta|$;
on side $s$, $\nu(\zeta)=\nu_s$ and $|d\zeta|=e^{-i\theta_s}d\zeta$, so
$\int_s\frac{\nu_s}{\zeta-z}f\,|d\zeta|
=\nu_s e^{-i\theta_s}\int_s\frac{f}{\zeta-z}d\zeta
=\nu_s e^{-i\theta_s}(2\pi i)\Cau_s f$, whose boundary value gives the first
formula in \eqref{eq:K-as-cauchy}. For $K_0'$, $\nu(z)$ is independent of
$\zeta$ and pulls out of the integral, leaving the per-side sum
$\sum_s e^{-i\theta_s}(2\pi i)\Cau_s f$, which gives the second formula. For
complex $f$, split
$\Real\frac{\nu}{\zeta-z}=\tfrac12(\frac{\nu}{\zeta-z}
+\frac{\bar\nu}{\overline{\zeta-z}})$; the first half yields $\Cau_{\Gw}$-type
boundary values and the second
$\overline{\Cau_{\Gw}\bar f}$-type, with side-constant coefficients.
\end{proof}

\begin{remark}
On a single straight side the self-interaction double-layer kernel vanishes,
since $\nu_s\perp(\zeta-z)$ for $\zeta,z$ collinear makes
$\Real\frac{\nu_s}{\zeta-z}=0$. The singular content of $K_0,K_0'$ at the vertex
is therefore the cross-side coupling, which is exactly the side-2 Cauchy
boundary value evaluated at the rotated argument $e^{-i\theta}z$---the corner
interaction the finite-sheeted calculus resolves.
\end{remark}

\subsection{Local decomposition at a rational vertex}
Let $\Gamma$ be piecewise analytic near a vertex $v$ with rational opening angle
$\theta_v=p_v\pi/q_v$. After translating $v$ to $0$, rotating, and rescaling, a
neighbourhood of $v$ is identified to leading order with a rational wedge
$\Gamma_{\theta_v,c_v}$; let $\sigma$ restrict to side densities lying in
conormal weighted spaces satisfying Assumption~\ref{ass:integrable} with angle
$\theta_v$. The Helmholtz layer operators are
$\mathcal S_k\sigma=\int_\Gamma\Phi_k\sigma\,ds$,
$\mathcal D_k\sigma=\int_\Gamma\partial_{n(y)}\Phi_k\,\sigma\,ds$,
$\mathcal D_k'\sigma=\int_\Gamma\partial_{n(x)}\Phi_k\,\sigma\,ds$.

\begin{theorem}[Local singular decomposition, kernel level]\label{thm:helmholtz}
Under the above hypotheses there is a neighbourhood of $v$ on which the
following identities of kernels hold:
\begin{equation}\label{eq:helmholtz-decomp}
  \mathcal S_k\sigma=-\tfrac{1}{2\pi}S_0\sigma+\mathcal R_{S,k}\sigma,\quad
  \mathcal D_k\sigma=-\tfrac{1}{2\pi}K_0\sigma+\mathcal R_{D,k}\sigma,\quad
  \mathcal D_k'\sigma=-\tfrac{1}{2\pi}K_0'\sigma+\mathcal R_{D',k}\sigma,
\end{equation}
where $S_0,K_0,K_0'$ are the rational-wedge model operators of
Proposition~\ref{prop:potential-algebra} (hence finite-sheeted Cauchy/logarithmic
boundary operators), and:
\begin{enumerate}
\item the single-layer remainder kernel is
$-\tfrac{1}{2\pi}(A_k-1)\log|x-y|+B_k=O(|x-y|^{2}\log|x-y|)+\text{analytic}$,
hence continuous across the diagonal;
\item the differentiated remainders $\mathcal R_{D,k},\mathcal R_{D',k}$ are
weakly singular (locally integrable), one order less singular than the
Cauchy-type corner kernels; and they are analytic away from the diagonal,
analytic in the kernel's analytic part when $\Gamma$ is analytic near $v$.
\end{enumerate}
The statement is at the level of kernel singularity: we identify the singular
part of each operator with a finite-sheeted wedge model and bound the order of
the remainder kernel, but we do not assert mapping or boundedness estimates for
the remainders on specified function spaces, in keeping with the caution of the
closing remark.
\end{theorem}

\begin{proof}
Insert \eqref{eq:phi-local} into $\mathcal S_k$. The leading piece
$-\tfrac{1}{2\pi}\int_\Gamma\log|x-y|\sigma\,ds$ is, in local wedge coordinates,
$-\tfrac{1}{2\pi}S_0\sigma=-\tfrac{1}{2\pi}\Lop_{\Gamma_{\theta_v,c_v}}\sigma$,
whose finite-sheeted structure and boundary trace are
Theorem~\ref{thm:log} and Proposition~\ref{prop:realpart}. The remaining kernel
$-\tfrac{1}{2\pi}(A_k-1)\log|x-y|+B_k$ is $O(|x-y|^{2}\log|x-y|)$ plus analytic,
which is (1). For $\mathcal D_k,\mathcal D_k'$, the normal at $v$ is piecewise
constant, so by \eqref{eq:normal-kernels} the leading normal-derivative kernels
are $\pm\Real\frac{\nu}{\zeta-z}$, i.e. $-\tfrac{1}{2\pi}K_0,-\tfrac{1}{2\pi}K_0'$
in the normalization of \eqref{eq:phi-local} after the factor
$A_k(0)=1$; Proposition~\ref{prop:potential-algebra} identifies these as
finite-sheeted Cauchy boundary operators. One normal derivative of the smoother
factor $(A_k-1)\log|x-y|+B_k$ lowers the order by one, leaving a weakly singular
remainder, which is (2).
\end{proof}

\subsection{Comparison with classical corner treatments}
Classical boundary integral treatments of nonsmooth domains handle corners
numerically, through mesh grading near vertices, specialized singular
quadrature, or local asymptotic enrichment; Kress's surveys of time-harmonic
acoustic scattering and linear integral equations are standard references for
the formulations and the graded-mesh handling of corners
\cite{kress1991boundary}. The present results are of a different character. They
do not propose a quadrature rule or a discretization. Instead they give an exact
local symbolic calculus: at a rational corner the singular part of each layer
operator is identified, modulo a strictly smoother remainder, with a
finite-sheeted rational-wedge model operator whose branch structure and corner
asymptotics are explicit (Theorems~\ref{thm:propagation}, \ref{thm:helmholtz}).
This is a local operator factorization, complementary to---and insertable
into---existing smooth-panel or Nystr\"om frameworks, rather than a substitute
for them. It is also complementary to the functional-analytic theory of layer
potentials on Lipschitz and nonsmooth domains
\cite{buffa2006acoustic,ammari2009layer,gao1991layer,kenig2011layer}, which provides the mapping
and invertibility framework on global spaces but not the explicit local
symbolic form of the singular part that the wedge calculus supplies. The
closest precedent is the Mellin-transform treatment of the double-layer
potential on polygons \cite{costabel1983normal}: there the corner enters
through a Mellin symbol, whereas here the finite covering $w=\zeta^{q}$
linearizes the angular monodromy and the singular part is read off the
finite-sheeted model directly.

\section{Worked examples}
\label{sec:examples}

We illustrate the factorization and the propagation calculus on the smallest
nontrivial covering, $q=2$, and check the resonance rule directly against exact
interval-transform formulae.

\subsection{The right angle \texorpdfstring{$\theta=\pi/2$}{theta=pi/2}}
Here $\theta=p\pi/q$ with $p=1$, $q=2$, so $(-1)^{p}=-1$ and
$e^{-i\theta}=-i$. The lifted intervals are $(0,1)$ and $(0,c^{2})$, the
evaluation points $z^{2}$ and $-z^{2}$, and the lifted densities
\eqref{eq:gjm} are $g_{j,0}(s)=f_j(s^{1/2})s^{-1/2}$ and
$g_{j,1}(s)=f_j(s^{1/2})$. The factorization \eqref{eq:main} reads
\begin{equation}\label{eq:rightangle}
  \Cau_{\Gamma_{\pi/2,c}}f(z)
  =\tfrac12\Bigl[z\,\Cau_{(0,1)}[g_{1,0}](z^{2})+\Cau_{(0,1)}[g_{1,1}](z^{2})\Bigr]
  +\tfrac12\Bigl[(-iz)\,\Cau_{(0,c^{2})}[g_{2,0}](-z^{2})
     +\Cau_{(0,c^{2})}[g_{2,1}](-z^{2})\Bigr].
\end{equation}
The two-sheeted structure is explicit: each side contributes one ``even'' block
($m=1$, no prefactor) and one ``odd'' block ($m=0$, prefactor linear in $z$),
both evaluated at the squared variable.

\subsection{A resonant smooth density produces a logarithm}
Take $f_1\equiv1$ on the first side and $f_2\equiv0$; the exponent is
$\alpha_{1,0}=0\in\Z_{\ge0}$, a resonance with $h=0$. Directly,
\begin{equation}\label{eq:const-exact}
  \Cau_{(0,1)}[1](z)
  =\frac{1}{2\pi i}\int_0^1\frac{dr}{r-z}
  =\frac{1}{2\pi i}\bigl(\log(1-z)-\log(-z)\bigr),
\end{equation}
which is exact. As $z\to0$, $\log(1-z)$ is analytic, so
\[
  \Cau_{\Gamma_{\pi/2,c}}f(z)=\Cau_{(0,1)}[1](z)
  =\frac{-1}{2\pi i}\,\Log(-z)+\bigl(\text{analytic}\bigr).
\]
This matches Lemma~\ref{lem:polyhom-model}(2) with $n_0=0$, $h=0$, whose leading
coefficient is $\frac{-1}{2\pi i\,(h+1)}=\frac{-1}{2\pi i}$. A perfectly smooth
density on the wedge thus produces a logarithmic corner singularity: the
resonance is geometric, generated by the integer exponent $\alpha=0$ meeting the
vertex, not by any singularity of the data.

\subsection{Logarithmic data raises the power}
Take $f_1(r)=\log r$, so $\alpha_{1,0}=0$ and $h=1$. By
Lemma~\ref{lem:polyhom-model}(2) the top logarithmic power rises to $h+1=2$,
with leading coefficient $\frac{-1}{2\pi i\,(h+1)}=\frac{-1}{4\pi i}$:
\[
  \Cau_{(0,1)}[\log r](z)\sim\frac{-1}{4\pi i}\,\bigl(\Log(-z)\bigr)^{2}
   +O\bigl(\Log(-z)\bigr).
\]
One verifies this from $\Cau_{(0,1)}[\log r]=\partial_\alpha
\Cau_{(0,1)}[r^{\alpha}]\big|_{\alpha=0}$: in the notation of the proof of
Lemma~\ref{lem:polyhom-model}, with $n_0=0$, $a=1$ (so $\ln a=0$), the resonant
pair is $G=\frac{-1}{2\pi i}\frac{e^{\eta L}-1}{\eta}+\cdots
=\frac{-1}{2\pi i}\bigl(L+\tfrac{\eta}{2}L^{2}+\cdots\bigr)$, and
$\partial_\alpha G|_{\alpha=0}=\frac{-1}{4\pi i}L^{2}+\cdots$ with
$L=\Log(-z)$. This is the log-power increase $h\mapsto h+1$ in its simplest instance.

\subsection{A non-resonant power gives a clean algebraic mode}
Take $f_1(r)=r^{-1/2}$ (so $\Real\alpha>-1$), $\alpha_{1,0}=-\tfrac12
\notin\Z_{\ge0}$, $h=0$. By Lemma~\ref{lem:model},
\[
  \Cau_{(0,1)}[r^{-1/2}](z)\sim m(-\tfrac12)(-z)^{-1/2}+(\text{analytic}),
  \qquad
  m(-\tfrac12)=\frac{1}{2i\sin(\pi/2)}=\frac{1}{2i},
\]
a purely algebraic corner mode with no logarithm, in contrast to the resonant
cases above. The next term, from the regular series of \eqref{eq:model}, is the
constant $\frac{1}{2\pi i}\frac{1}{\alpha}=\frac{i}{\pi}$ at $n=0$.

\subsection{Higher coverings}
For $q=3$ and $\theta=2\pi/3$ ($p=2$, even, $(-1)^{p}=+1$) the transform is a
sum of three blocks per side, on the intervals $(0,1)$ and $(0,c^{3})$, all
evaluated at $z^{3}$, with prefactors $z^{2},z,1$ on side $1$ and
$(e^{-2\pi i/3}z)^{2},(e^{-2\pi i/3}z),1$ on side $2$; the lifted exponents are
$\beta_{j,m}=(\alpha_j+m+1)/3-1$, $m=0,1,2$. A density mode $r^{\alpha_j}$ splits
into three sheet modes of exponents $(\alpha_j+1)/3-1$, $(\alpha_j+2)/3-1$,
$\alpha_j/3$, which recombine to $z^{\alpha_j}$ by the identity in the remark
following Theorem~\ref{thm:propagation}. The mechanism is identical to the
right-angle case, with $q=3$ sheets in place of $2$.

\section{Localization on rational polygons}
\label{sec:localization}

The calculus of \S\S\ref{sec:factor}--\ref{sec:helmholtz} is local to one
vertex. We now show that it assembles, with no new singular analysis, into a
global decomposition of the Cauchy operator on a piecewise analytic curve whose
corners are rational. Two structural facts drive the assembly: the Cauchy kernel
is conjugated \emph{exactly} by affine maps, and it is conjugated up to a
\emph{smoothing} (analytic-kernel) remainder by the analytic flattening of a
curved analytic side to its tangent ray. Separated vertex supports then make all
inter-vertex coupling analytic.

\subsection{Analytic corner charts and cutoffs}
Let $\Gamma\subset\C$ be a piecewise analytic Jordan curve with finite vertex
set $V(\Gamma)$, each interior opening angle a rational multiple of $\pi$,
$\theta_v=p_v\pi/q_v$, $(p_v,q_v)=1$. Near $v$, $\Gamma$ is the union of two
analytic arcs $\gamma_{v,1},\gamma_{v,2}$ issuing from $v$ with distinct unit
tangents. Fix a partition of unity
\begin{equation}\label{eq:pou}
  1=\chi_0+\sum_{v\in V(\Gamma)}\chi_v,
\end{equation}
where $\chi_v\in C^\infty(\Gamma)$ is supported in a vertex neighbourhood $U_v$
containing no other vertex, the $U_v$ are pairwise disjoint, and
$\chi_0=1-\sum_v\chi_v$ is supported away from all vertices on the analytic part
of $\Gamma$. For each $v$ let
\begin{equation}\label{eq:affine}
  A_v(z)=\lambda_v e^{-i\phi_v}(z-v)
\end{equation}
be the affine normalization sending $v\mapsto0$ and the first incident tangent
to the positive real axis, so that the two tangent rays of $A_v\gamma_{v,1},
A_v\gamma_{v,2}$ are the sides of the model wedge $\Gamma_{\theta_v,c_v}$ for a
suitable $c_v>0$. Let $\sigma_v$ denote the side-wise analytic flattening map:
on each side it is the biholomorphism, defined near $0$ with $\sigma_v(0)=0$,
$\sigma_v'(0)=1$, carrying the straight model side onto the (affinely
normalized) analytic arc $A_v\gamma_{v,j}$. Figure~\ref{fig:boundary_localization} illustrates the localization scheme for a curvilinear polygon, where cutoff regions reduce the global boundary analysis to local vertex models.

\begin{figure}[htbp]
    \centering
    \includegraphics[width=0.75\textwidth]{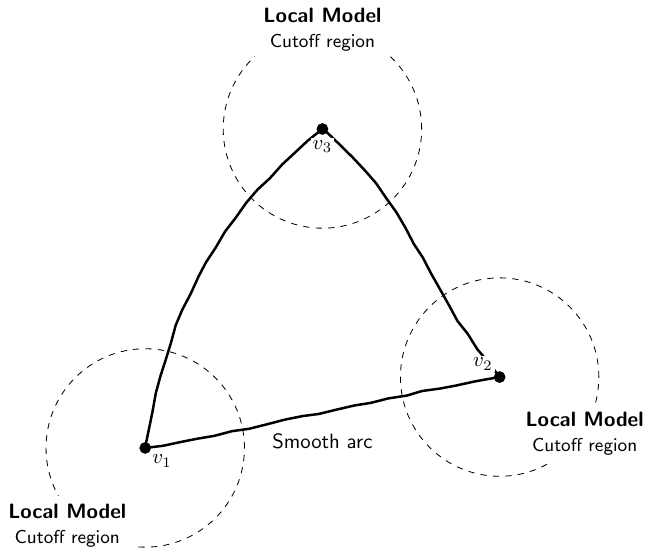}
    \caption{Localization scheme for a curvilinear polygon. Dashed circles indicate the cutoff regions near the vertices $v_j$, isolating the local wedge singularities from the smooth boundary arcs.}
    \label{fig:boundary_localization}
\end{figure}

\subsection{Transport of the wedge model}
The next two lemmas are the engine of the localization.

\begin{lemma}[Exact affine conjugation]\label{lem:affine}
For any integrable density $h$ on $\Gamma$ and any affine
$A(z)=\lambda e^{-i\phi}(z-v)$,
\[
  \Cau_{\Gamma}[h]=\bigl(\Cau_{A\Gamma}[h\circ A^{-1}]\bigr)\circ A
  \qquad\text{on }\C\setminus\Gamma .
\]
The conjugation is exact: no remainder arises.
\end{lemma}

\begin{proof}
With $w=A(z)$, $\eta=A(\zeta)$, one has $A^{-1}(\eta)=v+\lambda^{-1}e^{i\phi}\eta$,
so $\zeta-z=\lambda^{-1}e^{i\phi}(\eta-w)$ and $d\zeta=\lambda^{-1}e^{i\phi}d\eta$.
The factors cancel in $\frac{h(\zeta)\,d\zeta}{\zeta-z}
=\frac{(h\circ A^{-1})(\eta)}{\eta-w}\,d\eta$, giving the claim.
\end{proof}

\begin{lemma}[Analytic flattening with smoothing remainder]\label{lem:flatten}
Let $\ell$ be a ray from $0$, $\sigma$ a biholomorphism near $0$ with
$\sigma(0)=0$, $\sigma'(0)\neq0$, and $\tilde\gamma=\sigma(\ell)$ the
corresponding analytic arc. For $g$ integrable on $\tilde\gamma$ and $w_0$ near
$0$,
\[
  \Cau_{\tilde\gamma}[g]\bigl(\sigma(w_0)\bigr)
  =\Cau_{\ell}\bigl[(g\circ\sigma)\,\sigma'\bigr](w_0)
   +\mathcal K_\sigma[g](w_0),
\]
where $\mathcal K_\sigma[g](w_0)=\frac{1}{2\pi i}\int_\ell
(g\circ\sigma)(\tau)\sigma'(\tau)K_\sigma(\tau,w_0)\,d\tau$ has kernel
$K_\sigma$ analytic in a neighbourhood of $(0,0)$, with diagonal value
$K_\sigma(w_0,w_0)=\sigma''(w_0)/(2\sigma'(w_0))$. Hence $\mathcal K_\sigma$ maps
$g$ to a function analytic at $0$: it is smoothing across the vertex.
\end{lemma}

\begin{proof}
Parametrize $\eta=\sigma(\tau)$, $\tau\in\ell$, and set the evaluation point
$\sigma(w_0)$. Then
$\Cau_{\tilde\gamma}[g](\sigma(w_0))=\frac{1}{2\pi i}\int_\ell
\frac{g(\sigma(\tau))\sigma'(\tau)}{\sigma(\tau)-\sigma(w_0)}\,d\tau$. Writing
$\frac{\sigma'(\tau)}{\sigma(\tau)-\sigma(w_0)}=\frac{1}{\tau-w_0}
+K_\sigma(\tau,w_0)$ and using
$\sigma(\tau)-\sigma(w_0)=\sigma'(w_0)(\tau-w_0)\bigl(1+\tfrac{\sigma''(w_0)}
{2\sigma'(w_0)}(\tau-w_0)+\cdots\bigr)$ shows the simple pole at $\tau=w_0$ has
residue $1$, so $K_\sigma$ extends analytically across the diagonal with
$K_\sigma(w_0,w_0)=\sigma''(w_0)/(2\sigma'(w_0))$; $K_\sigma$ is analytic wherever
$\sigma$ is analytic and $\sigma'\neq0$, i.e. near $(0,0)$. The first term is the
flat-ray Cauchy transform of the transported density
$(g\circ\sigma)\sigma'$. The remainder $\mathcal K_\sigma[g]$ has an analytic
kernel, so it is a holomorphic function of $w_0$ near $0$.
\end{proof}

\begin{proposition}[Vertex model with analytic remainder]\label{prop:vertex-model}
Define the vertex model operator
\[
  \Cau_v[f]:=\bigl(\Cau_{\Gamma_{\theta_v,c_v}}\bigl[\Psi_v(\chi_v f)\bigr]\bigr)
              \circ\bigl(\sigma_v^{-1}\!\circ A_v\bigr),
  \qquad
  \Psi_v(\chi_v f)=\bigl((\chi_v f)\circ A_v^{-1}\circ\sigma_v\bigr)\,\sigma_v',
\]
where $\Psi_v$ transports the localized density onto the straight model sides.
Then near $v$,
\[
  \Cau_\Gamma[\chi_v f]=\Cau_v[f]+(\text{analytic at }v).
\]
\end{proposition}

\begin{proof}
By Lemma~\ref{lem:affine}, $\Cau_\Gamma[\chi_v f]=\Cau_{A_v\Gamma}
[(\chi_v f)\circ A_v^{-1}]\circ A_v$, and $A_v\Gamma$ near $0$ is the union of the
analytic arcs $\tilde\gamma_{v,j}=\sigma_v(\text{model side }j)$. Applying
Lemma~\ref{lem:flatten} on each side to $g=(\chi_v f)\circ A_v^{-1}$ replaces
$\Cau_{A_v\Gamma}[g]$ by $\Cau_{\Gamma_{\theta_v,c_v}}[(g\circ\sigma_v)\sigma_v']
=\Cau_{\Gamma_{\theta_v,c_v}}[\Psi_v(\chi_v f)]$, evaluated at
$\sigma_v^{-1}(A_v\,\cdot\,)$, plus the analytic remainders $\mathcal K_{\sigma_v}$.
This is the assertion.
\end{proof}

\subsection{Smoothness of off-vertex interactions}
\begin{lemma}[Separated supports]\label{lem:separated}
For $v\neq w$ in $V(\Gamma)$, $\Cau_\Gamma[\chi_w f]$ is analytic in a
neighbourhood of $v$; likewise $\Cau_\Gamma[\chi_0 f]$ is analytic near every
vertex.
\end{lemma}

\begin{proof}
The kernel $\frac{1}{\zeta-z}$ with $\zeta\in\operatorname{supp}\chi_w$ (or
$\operatorname{supp}\chi_0$) and $z$ in a neighbourhood of $v$ disjoint from that
support is bounded with $|\zeta-z|\ge d>0$ and depends holomorphically on $z$;
differentiation under the integral gives holomorphy of $\Cau_\Gamma[\chi_w f]$
in $z$ near $v$.
\end{proof}

\subsection{The global decomposition and its asymptotic consequence}

\begin{theorem}[Local-to-global singular decomposition on rational polygons]
\label{thm:localization}
Let $\Gamma$ be a piecewise analytic Jordan curve with rational corner angles
$\theta_v=p_v\pi/q_v$, $v\in V(\Gamma)$, and cutoffs \eqref{eq:pou}. Then
\begin{equation}\label{eq:localization}
  \Cau_\Gamma
  =\sum_{v\in V(\Gamma)}\Cau_v
   +\Cau_\Gamma^{\mathrm{sm}}
   +\mathcal R_\Gamma,
\end{equation}
where each $\Cau_v$ is the transported finite-sheeted rational-wedge operator of
Proposition~\ref{prop:vertex-model} for the model angle $\theta_v$,
$\Cau_\Gamma^{\mathrm{sm}}=\Cau_\Gamma\chi_0$ is the smooth-arc Cauchy operator
(singular only along the diagonal away from the vertices), and the remainder
$\mathcal R_\Gamma$ has a kernel analytic in a neighbourhood of every vertex. In
particular, near each vertex $v$,
\[
  \Cau_\Gamma f-\Cau_v[f]\quad\text{is analytic},
\]
so the singular structure of $\Cau_\Gamma$ at $v$ is exactly that of the
finite-sheeted wedge model $\Cau_{\Gamma_{\theta_v,c_v}}$.
\end{theorem}

\begin{proof}
Decompose $f=\chi_0 f+\sum_v\chi_v f$ by \eqref{eq:pou}, so
\[
  \Cau_\Gamma f
  =\Cau_\Gamma[\chi_0 f]+\sum_v\Cau_\Gamma[\chi_v f]
  =\Cau_\Gamma^{\mathrm{sm}}f+\sum_v\Cau_v[f]+\sum_v R_v[f],
\]
where $\Cau_\Gamma^{\mathrm{sm}}:=\Cau_\Gamma[\chi_0\,\cdot\,]$ and
\[
  R_v[f]:=\Cau_\Gamma[\chi_v f]-\Cau_v[f].
\]
Set $\mathcal R_\Gamma:=\sum_v R_v$; no term is counted twice. Each $R_v$ is
analytic near $v$ by Proposition~\ref{prop:vertex-model} (the flattening
remainder), and analytic near every other vertex $w\neq v$ by
Lemma~\ref{lem:separated} (separated supports, since $\operatorname{supp}\chi_v$
is bounded away from $w$). Hence $\mathcal R_\Gamma=\sum_v R_v$ is analytic in a
neighbourhood of every vertex. The smooth-arc term $\Cau_\Gamma^{\mathrm{sm}}$ is
likewise analytic near each vertex (its source $\chi_0 f$ is supported away from
the vertices, Lemma~\ref{lem:separated}), while retaining the ordinary Plemelj
singularity along the diagonal on the analytic part of $\Gamma$, where the
vertex models are not singular. Thus $\mathcal R_\Gamma$ (and
$\Cau_\Gamma^{\mathrm{sm}}$) is smooth precisely across the vertex-localized
singular supports, which is \eqref{eq:localization}.
\end{proof}

\begin{corollary}[Vertexwise polyhomogeneous expansion]\label{cor:vertexwise}
Suppose $f$ has classical conormal polyhomogeneous expansions at every rational
corner of $\Gamma$. Then $\Cau_\Gamma f$ has, at each vertex $v$, the sectorwise
expansion of Theorem~\ref{thm:propagation} for the model angle $\theta_v$,
computed from the transported density $\Psi_v(\chi_v f)$. No new singular
exponents are generated by the curvature of the sides or by inter-vertex
interaction: the analytic flattening $\sigma_v$ and the analytic remainder
$\mathcal R_\Gamma$ contribute only integer powers, which are absorbed into the
analytic family of the expansion.
\end{corollary}

\begin{proof}
By Theorem~\ref{thm:localization}, near $v$ the singular part of $\Cau_\Gamma f$
equals that of $\Cau_v[f]$, a transported wedge model. The transport maps
$\Psi_v$ and $\sigma_v^{-1}\!\circ A_v$ are analytic with nonvanishing
derivative at $0$, so they preserve conormal polyhomogeneity: a mode
$r^{\alpha}(\log r)^{h}$ has $(\sigma_v(r))^{\alpha}=\sigma_v'(0)^{\alpha}
r^{\alpha}(1+O(r))$ and $\log\sigma_v(r)=\log r+O(1)$, leaving the exponents
$\alpha$ and the top log power unchanged while shifting only the coefficients.
Theorem~\ref{thm:propagation} then gives the sectorwise expansion of
$\Cau_v[f]$; the remainder $\mathcal R_\Gamma$ is analytic at $v$ and adds only
to the integer-power analytic family.
\end{proof}

\begin{remark}
Theorem~\ref{thm:localization} promotes the wedge calculus to a rational-polygon
calculus: the singular part of the Cauchy operator on $\Gamma$ is a finite sum
of finite-sheeted wedge models, one per corner, the corners not interacting at
the level of singular structure. The same localization applies verbatim to the
logarithmic and Helmholtz layer operators of \S\S\ref{sec:log}--\ref{sec:helmholtz}
through Lemmas~\ref{lem:affine}--\ref{lem:separated}, since those operators were
already reduced there to the wedge Cauchy and logarithmic transforms. We do not
develop a global Fredholm theory, which would require mapping properties of
$\mathcal R_\Gamma$ and of the corner blocks on a global space; the present
statement is at the level of kernel regularity and conormal asymptotics,
matching the caution of Theorem~\ref{thm:helmholtz}.
\end{remark}

\section{Outlook: irrational angles}
\label{sec:outlook}

The construction is genuinely arithmetic in the angle. The factorization
\eqref{eq:factor-thm} and the propagation Theorem~\ref{thm:propagation} use that
$\theta=p\pi/q$ with $q$ finite, so that $w=\zeta^{q}$ is a finite covering and
the angular monodromy of $z^{\alpha}$ closes after $q$ sheets; the polygon
localization of \S\ref{sec:localization} then assembles the finite-sheeted
models across all corners. For irrational $\theta/\pi$ no finite covering
linearizes the corner: the natural object is an infinite-sheeted covering, or a
limit of rational models $p_n/q_n\to\theta/\pi$ with $q_n\to\infty$. The conormal
mapping bound of Theorem~\ref{thm:Lq-bounded} has operator norm $q_n\to\infty$
and the H\"older exponent $\beta/q_n\to0$, so the rational models do not pass
uniformly to the limit; an irrational-angle theory would require a genuinely
different, non-finite-sheeted mechanism. We therefore state the present results
only for rational angles, where the finite-sheeted models are exact, and leave
the irrational case open.

\section*{Author’s Important Statements}
All authors are satisfied with this preprint. This research is the joint work with Jiguang Yu from Boston University, USA and Ye Liang from the University of Iowa, USA. Some of the work comes from Ye Liang's undergraduate research at University College London, UK.)



\begin{thebibliography}{56}
\ifx \bisbn   \undefined \def \bisbn  #1{ISBN #1}\fi
\ifx \binits  \undefined \def \binits#1{#1}\fi
\ifx \bauthor  \undefined \def \bauthor#1{#1}\fi
\ifx \batitle  \undefined \def \batitle#1{#1}\fi
\ifx \bjtitle  \undefined \def \bjtitle#1{#1}\fi
\ifx \bvolume  \undefined \def \bvolume#1{\textbf{#1}}\fi
\ifx \byear  \undefined \def \byear#1{#1}\fi
\ifx \bissue  \undefined \def \bissue#1{#1}\fi
\ifx \bfpage  \undefined \def \bfpage#1{#1}\fi
\ifx \blpage  \undefined \def \blpage #1{#1}\fi
\ifx \burl  \undefined \def \burl#1{\textsf{#1}}\fi
\ifx \doiurl  \undefined \def \doiurl#1{\url{https://doi.org/#1}}\fi
\ifx \betal  \undefined \def \betal{\textit{et al.}}\fi
\ifx \binstitute  \undefined \def \binstitute#1{#1}\fi
\ifx \binstitutionaled  \undefined \def \binstitutionaled#1{#1}\fi
\ifx \bctitle  \undefined \def \bctitle#1{#1}\fi
\ifx \beditor  \undefined \def \beditor#1{#1}\fi
\ifx \bpublisher  \undefined \def \bpublisher#1{#1}\fi
\ifx \bbtitle  \undefined \def \bbtitle#1{#1}\fi
\ifx \bedition  \undefined \def \bedition#1{#1}\fi
\ifx \bseriesno  \undefined \def \bseriesno#1{#1}\fi
\ifx \blocation  \undefined \def \blocation#1{#1}\fi
\ifx \bsertitle  \undefined \def \bsertitle#1{#1}\fi
\ifx \bsnm \undefined \def \bsnm#1{#1}\fi
\ifx \bsuffix \undefined \def \bsuffix#1{#1}\fi
\ifx \bparticle \undefined \def \bparticle#1{#1}\fi
\ifx \barticle \undefined \def \barticle#1{#1}\fi
\bibcommenthead
\ifx \bconfdate \undefined \def \bconfdate #1{#1}\fi
\ifx \botherref \undefined \def \botherref #1{#1}\fi
\ifx \url \undefined \def \url#1{\textsf{#1}}\fi
\ifx \bchapter \undefined \def \bchapter#1{#1}\fi
\ifx \bbook \undefined \def \bbook#1{#1}\fi
\ifx \bcomment \undefined \def \bcomment#1{#1}\fi
\ifx \oauthor \undefined \def \oauthor#1{#1}\fi
\ifx \citeauthoryear \undefined \def \citeauthoryear#1{#1}\fi
\ifx \endbibitem  \undefined \def \endbibitem {}\fi
\ifx \bconflocation  \undefined \def \bconflocation#1{#1}\fi
\ifx \arxivurl  \undefined \def \arxivurl#1{\textsf{#1}}\fi
\csname PreBibitemsHook\endcsname

\bibitem[\protect\citeauthoryear{Kerzman and Stein}{1978}]{kerzman1978cauchy}
\begin{barticle}
\bauthor{\bsnm{Kerzman}, \binits{N.}},
\bauthor{\bsnm{Stein}, \binits{E.M.}}:
\batitle{The {C}auchy kernel, the {S}zeg{\"o} kernel, and the {R}iemann mapping function}.
\bjtitle{Mathematische Annalen}
\bvolume{236}(\bissue{1}),
\bfpage{85}--\blpage{93}
(\byear{1978})
\end{barticle}
\endbibitem

\bibitem[\protect\citeauthoryear{Bolt and Raich}{2015}]{bolt2015kerzman}
\begin{barticle}
\bauthor{\bsnm{Bolt}, \binits{M.}},
\bauthor{\bsnm{Raich}, \binits{A.}}:
\batitle{The {K}erzman--{S}tein operator for piecewise continuously differentiable regions}.
\bjtitle{Complex Variables and Elliptic Equations}
\bvolume{60}(\bissue{4}),
\bfpage{478}--\blpage{492}
(\byear{2015})
\end{barticle}
\endbibitem

\bibitem[\protect\citeauthoryear{King}{2009}]{king2009hilbert}
\begin{bbook}
\bauthor{\bsnm{King}, \binits{F.W.}}:
\bbtitle{Hilbert {T}ransforms: {V}olume 2}
vol. \bseriesno{2}.
\bpublisher{Cambridge University Press},
\blocation{Cambridge}
(\byear{2009})
\end{bbook}
\endbibitem

\bibitem[\protect\citeauthoryear{Olver}{2011}]{olver2011computing}
\begin{barticle}
\bauthor{\bsnm{Olver}, \binits{S.}}:
\batitle{Computing the {H}ilbert transform and its inverse}.
\bjtitle{Mathematics of computation}
\bvolume{80}(\bissue{275}),
\bfpage{1745}--\blpage{1767}
(\byear{2011})
\end{barticle}
\endbibitem

\bibitem[\protect\citeauthoryear{Reyes and Blaya}{2003}]{reyes2003one}
\begin{barticle}
\bauthor{\bsnm{Reyes}, \binits{J.B.}},
\bauthor{\bsnm{Blaya}, \binits{R.A.}}:
\batitle{One-dimensional singular integral equations}.
\bjtitle{Complex Variables Theory and Application}
\bvolume{48}(\bissue{6}),
\bfpage{483}--\blpage{493}
(\byear{2003})
\end{barticle}
\endbibitem

\bibitem[\protect\citeauthoryear{Muskhelishvili and Radok}{2008}]{muskhelishvili2008singular}
\begin{bbook}
\bauthor{\bsnm{Muskhelishvili}, \binits{N.I.}},
\bauthor{\bsnm{Radok}, \binits{J.R.M.}}:
\bbtitle{Singular Integral Equations: Boundary Problems of Function Theory and Their Application to Mathematical Physics}.
\bpublisher{Courier Corporation},
\blocation{Mineola, NY}
(\byear{2008})
\end{bbook}
\endbibitem

\bibitem[\protect\citeauthoryear{Gakhov}{2014}]{gakhov2014boundary}
\begin{bbook}
\bauthor{\bsnm{Gakhov}, \binits{F.D.}}:
\bbtitle{Boundary Value Problems}.
\bpublisher{Elsevier},
\blocation{Amsterdam}
(\byear{2014})
\end{bbook}
\endbibitem

\bibitem[\protect\citeauthoryear{Wang and Yu}{2025}]{wang2025analysis}
\begin{barticle}
\bauthor{\bsnm{Wang}, \binits{L.S.}},
\bauthor{\bsnm{Yu}, \binits{J.}}:
\batitle{Analysis framework for stochastic predator--prey model with demographic noise}.
\bjtitle{Results in Applied Mathematics}
\bvolume{27},
\bfpage{100621}
(\byear{2025})
\end{barticle}
\endbibitem

\bibitem[\protect\citeauthoryear{Wang et~al.}{2025}]{wang2025analysis1}
\begin{barticle}
\bauthor{\bsnm{Wang}, \binits{L.S.}},
\bauthor{\bsnm{Yu}, \binits{J.}},
\bauthor{\bsnm{Li}, \binits{S.}},
\bauthor{\bsnm{Liu}, \binits{Z.}}:
\batitle{Analysis and mean-field limit of a hybrid pde-abm modeling angiogenesis-regulated resistance evolution}.
\bjtitle{Mathematics}
\bvolume{13}(\bissue{17}),
\bfpage{2898}
(\byear{2025})
\end{barticle}
\endbibitem

\bibitem[\protect\citeauthoryear{Pritsker}{2008}]{pritsker2008find}
\begin{barticle}
\bauthor{\bsnm{Pritsker}, \binits{I.}}:
\batitle{How to find a measure from its potential}.
\bjtitle{Computational Methods and Function Theory}
\bvolume{8}(\bissue{2}),
\bfpage{597}--\blpage{614}
(\byear{2008})
\end{barticle}
\endbibitem

\bibitem[\protect\citeauthoryear{Abreu~Blaya and Bory~Reyes}{2012}]{abreu2012plemelj}
\begin{barticle}
\bauthor{\bsnm{Abreu~Blaya}, \binits{R.}},
\bauthor{\bsnm{Bory~Reyes}, \binits{J.}}:
\batitle{The {P}lemelj--{P}rivalov theorem in variable exponent {C}lifford analysis}.
\bjtitle{Georgian Math. J}
\bvolume{19}(\bissue{3}),
\bfpage{401}--\blpage{415}
(\byear{2012})
\end{barticle}
\endbibitem

\bibitem[\protect\citeauthoryear{Yu et~al.}{2026}]{yu2026from}
\begin{barticle}
\bauthor{\bsnm{Yu}, \binits{J.}},
\bauthor{\bsnm{Wang}, \binits{L.S.}},
\bauthor{\bsnm{Ban}, \binits{S.}},
\bauthor{\bsnm{Liang}, \binits{Y.}}:
\batitle{From microscopic damage to macroscopic games: a dimensionality reduction of stem cell homeostasis}.
\bjtitle{Transport Phenomena}
\bvolume{1}(\bissue{2}),
\bfpage{20260037}
(\byear{2026})
\end{barticle}
\endbibitem

\bibitem[\protect\citeauthoryear{Abreu-Blaya et~al.}{2007}]{abreu2007notion}
\begin{barticle}
\bauthor{\bsnm{Abreu-Blaya}, \binits{R.}},
\bauthor{\bsnm{Bory-Reyes}, \binits{J.}},
\bauthor{\bsnm{Shapiro}, \binits{M.}}:
\batitle{On the notion of the {B}ochner--{M}artinelli integral for domains with rectifiable boundary}.
\bjtitle{Complex analysis and Operator theory}
\bvolume{1}(\bissue{2}),
\bfpage{143}--\blpage{168}
(\byear{2007})
\end{barticle}
\endbibitem

\bibitem[\protect\citeauthoryear{Liu et~al.}{2025}]{liu2025bidirectional}
\begin{barticle}
\bauthor{\bsnm{Liu}, \binits{Z.}},
\bauthor{\bsnm{Wang}, \binits{L.S.}},
\bauthor{\bsnm{Yu}, \binits{J.}},
\bauthor{\bsnm{Zhang}, \binits{J.}},
\bauthor{\bsnm{Martel}, \binits{E.}},
\bauthor{\bsnm{Li}, \binits{S.}}:
\batitle{Bidirectional endothelial feedback drives turing-vascular patterning and drug-resistance niches: a hybrid pde-agent-based study}.
\bjtitle{Bioengineering}
\bvolume{12}(\bissue{10}),
\bfpage{1097}
(\byear{2025})
\end{barticle}
\endbibitem

\bibitem[\protect\citeauthoryear{Schrohe}{2023}]{schrohe2023introduction}
\begin{bchapter}
\bauthor{\bsnm{Schrohe}, \binits{E.}}:
\bctitle{Introduction to the {A}nalysis on {M}anifolds with {C}onical {S}ingularities}.
In: \bbtitle{Methusalem Workshop on Classical Analysis and PDEs},
pp. \bfpage{105}--\blpage{118}
(\byear{2023}).
\bcomment{Springer}
\end{bchapter}
\endbibitem

\bibitem[\protect\citeauthoryear{Liang et~al.}{2025}]{liang2025global}
\begin{barticle}
\bauthor{\bsnm{Liang}, \binits{Y.}},
\bauthor{\bsnm{Wang}, \binits{L.S.}},
\bauthor{\bsnm{Yu}, \binits{J.}},
\bauthor{\bsnm{Liu}, \binits{Z.}}:
\batitle{Global well-posedness and stability of nonlocal damage-structured lineage model with feedback and dedifferentiation}.
\bjtitle{Mathematics}
\bvolume{13}(\bissue{22}),
\bfpage{3583}
(\byear{2025})
\end{barticle}
\endbibitem

\bibitem[\protect\citeauthoryear{Schulze}{2022}]{schulze2022mellin}
\begin{bchapter}
\bauthor{\bsnm{Schulze}, \binits{B.-W.}}:
\bctitle{Mellin operators in weighted corner spaces}.
In: \bbtitle{Differential Equations on Manifolds and Mathematical Physics: Dedicated to the Memory of Boris Sternin},
pp. \bfpage{287}--\blpage{313}.
\bpublisher{Springer},
\blocation{Cham}
(\byear{2022})
\end{bchapter}
\endbibitem

\bibitem[\protect\citeauthoryear{Anjam}{2020}]{anjam2020geometric}
\begin{barticle}
\bauthor{\bsnm{Anjam}, \binits{Y.N.}}:
\batitle{Geometric {S}ingularities of the {P}oisson's {E}quation in a {N}on-{S}mooth {D}omain with {A}pplications of {W}eighted {S}obolev {S}paces}.
\bjtitle{International Journal of Analysis and Applications}
\bvolume{18}(\bissue{4}),
\bfpage{672}--\blpage{688}
(\byear{2020})
\end{barticle}
\endbibitem

\bibitem[\protect\citeauthoryear{Maday and Marcati}{2019}]{maday2019regularity}
\begin{barticle}
\bauthor{\bsnm{Maday}, \binits{Y.}},
\bauthor{\bsnm{Marcati}, \binits{C.}}:
\batitle{Regularity and hp discontinuous {G}alerkin finite element approximation of linear elliptic eigenvalue problems with singular potentials}.
\bjtitle{Mathematical Models and Methods in Applied Sciences}
\bvolume{29}(\bissue{08}),
\bfpage{1585}--\blpage{1617}
(\byear{2019})
\end{barticle}
\endbibitem

\bibitem[\protect\citeauthoryear{Wang and Yu}{2026}]{wang2026algebraic}
\begin{barticle}
\bauthor{\bsnm{Wang}, \binits{L.S.}},
\bauthor{\bsnm{Yu}, \binits{J.}}:
\batitle{Algebraic--spectral thresholds and discrete--continuous stability transfer in leslie--gower systems}.
\bjtitle{Electronic Research Archive}
\bvolume{34}(\bissue{1}),
\bfpage{251}--\blpage{290}
(\byear{2026})
\end{barticle}
\endbibitem

\bibitem[\protect\citeauthoryear{David and Journ{\'e}}{1984}]{david1984boundedness}
\begin{botherref}
\oauthor{\bsnm{David}, \binits{G.}},
\oauthor{\bsnm{Journ{\'e}}, \binits{J.-L.}}:
A boundedness criterion for generalized {C}alder{\'o}n-{Z}ygmund operators.
Annals of Mathematics,
371--397
(1984)
\end{botherref}
\endbibitem

\bibitem[\protect\citeauthoryear{Gao et~al.}{2022}]{gao2022rolling}
\begin{bchapter}
\bauthor{\bsnm{Gao}, \binits{Y.}},
\bauthor{\bsnm{Li}, \binits{L.}},
\bauthor{\bsnm{Yu}, \binits{J.}}:
\bctitle{Rolling prediction model of closing price based on eemd data noise reduction and hgs-delm}.
In: \bbtitle{2022 International Conference on Data Analytics, Computing and Artificial Intelligence (ICDACAI)},
pp. \bfpage{255}--\blpage{260}
(\byear{2022}).
\bcomment{IEEE}
\end{bchapter}
\endbibitem

\bibitem[\protect\citeauthoryear{Rochberg}{1984}]{rochberg1984calderon}
\begin{bchapter}
\bauthor{\bsnm{Rochberg}, \binits{R.}}:
\bctitle{Calder{\'o}n--{Z}ygmund operators, pseudo-differential operators and the {C}auchy integral of {C}alder{\'o}n}.
In: \bbtitle{Harmonic Analysis in Euclidean Spaces (Proceedings of Symposia in Pure Mathematics)}.
\bpublisher{American Mathematical Society},
\blocation{Providence, RI}
(\byear{1984})
\end{bchapter}
\endbibitem

\bibitem[\protect\citeauthoryear{Calder{\'o}n}{1977}]{calderon1977cauchy}
\begin{barticle}
\bauthor{\bsnm{Calder{\'o}n}, \binits{A.-P.}}:
\batitle{Cauchy integrals on {L}ipschitz curves and related operators}.
\bjtitle{Proceedings of the National Academy of Sciences}
\bvolume{74}(\bissue{4}),
\bfpage{1324}--\blpage{1327}
(\byear{1977})
\end{barticle}
\endbibitem

\bibitem[\protect\citeauthoryear{Wang et~al.}{2026}]{wang2026damage}
\begin{barticle}
\bauthor{\bsnm{Wang}, \binits{L.S.}},
\bauthor{\bsnm{Yu}, \binits{J.}},
\bauthor{\bsnm{Liu}, \binits{Z.}}:
\batitle{A damage-structured pde model of stem cell hierarchies: The dual role of dedifferentiation in tissue homeostasis and aging}.
\bjtitle{Plos one}
\bvolume{21}(\bissue{2}),
\bfpage{0335163}
(\byear{2026})
\end{barticle}
\endbibitem

\bibitem[\protect\citeauthoryear{McIntosh and Tao}{2006}]{mcintosh2006convolution}
\begin{bchapter}
\bauthor{\bsnm{McIntosh}, \binits{A.}},
\bauthor{\bsnm{Tao}, \binits{Q.}}:
\bctitle{Convolution singular integral operators on {L}ipschitz curves}.
In: \bbtitle{Harmonic Analysis: Proceedings of the Special Program at the Nankai Institute of Mathematics Tianjin, PR China, March--July, 1988},
pp. \bfpage{142}--\blpage{162}.
\bpublisher{Springer},
\blocation{Berlin Heidelberg}
(\byear{2006})
\end{bchapter}
\endbibitem

\bibitem[\protect\citeauthoryear{Wang et~al.}{2025}]{wang2025multi}
\begin{barticle}
\bauthor{\bsnm{Wang}, \binits{Z.}},
\bauthor{\bsnm{Wang}, \binits{D.}},
\bauthor{\bsnm{Yu}, \binits{J.}}:
\batitle{Multi-strategy hybrid improved intelligent algorithm for solving uav-mtsp}.
\bjtitle{Information Technology and Control}
\bvolume{54}(\bissue{2}),
\bfpage{413}--\blpage{438}
(\byear{2025})
\end{barticle}
\endbibitem

\bibitem[\protect\citeauthoryear{Coifman et~al.}{1982}]{coifman1982integrale}
\begin{barticle}
\bauthor{\bsnm{Coifman}, \binits{R.R.}},
\bauthor{\bsnm{McIntosh}, \binits{A.}},
\bauthor{\bsnm{Meyer}, \binits{Y.}}:
\batitle{L'int{\'e}grale de {C}auchy d{\'e}finit un op{\'e}rateur born{\'e} sur $l^2$ pour les courbes lipschitziennes}.
\bjtitle{Annals of Mathematics}
\bvolume{116}(\bissue{2}),
\bfpage{361}--\blpage{387}
(\byear{1982})
\end{barticle}
\endbibitem

\bibitem[\protect\citeauthoryear{Coifman et~al.}{1989}]{coifman1989two}
\begin{barticle}
\bauthor{\bsnm{Coifman}, \binits{R.R.}},
\bauthor{\bsnm{Jones}, \binits{P.W.}},
\bauthor{\bsnm{Semmes}, \binits{S.}}:
\batitle{Two elementary proofs of the $l^2$ boundedness of {C}auchy integrals on {L}ipschitz curves}.
\bjtitle{Journal of the American Mathematical Society}
\bvolume{2}(\bissue{3}),
\bfpage{553}--\blpage{564}
(\byear{1989})
\end{barticle}
\endbibitem

\bibitem[\protect\citeauthoryear{Yu et~al.}{2026}]{yu2026pattern}
\begin{botherref}
\oauthor{\bsnm{Yu}, \binits{J.}},
\oauthor{\bsnm{Wang}, \binits{L.S.}},
\oauthor{\bsnm{Liu}, \binits{Z.}},
\oauthor{\bsnm{Liu}, \binits{J.}}:
Pattern suppression and recovery under one-way versus two-way chemotactic coupling in hybrid partial differential equation--ordinary differential equation models.
Transport Phenomena
(0)
(2026)
\end{botherref}
\endbibitem

\bibitem[\protect\citeauthoryear{Qian et~al.}{2019}]{qian2019singular}
\begin{bbook}
\bauthor{\bsnm{Qian}, \binits{T.}},
\bauthor{\bsnm{Li}, \binits{P.}}, \betal:
\bbtitle{Singular Integrals and {F}ourier Theory on {L}ipschitz Boundaries}.
\bpublisher{Springer},
\blocation{Singapore}
(\byear{2019})
\end{bbook}
\endbibitem

\bibitem[\protect\citeauthoryear{Jones}{2006}]{jones2006square}
\begin{bchapter}
\bauthor{\bsnm{Jones}, \binits{P.W.}}:
\bctitle{Square functions, {C}auchy integrals, analytic capacity, and harmonic measure}.
In: \bbtitle{Harmonic Analysis and Partial Differential Equations: Proceedings of the International Conference Held in El Escorial, Spain, June 9--13, 1987},
pp. \bfpage{24}--\blpage{68}
(\byear{2006}).
\bcomment{Springer}
\end{bchapter}
\endbibitem

\bibitem[\protect\citeauthoryear{Cai et~al.}{2026}]{cai2026optimal}
\begin{botherref}
\oauthor{\bsnm{Cai}, \binits{J.}},
\oauthor{\bsnm{Chen}, \binits{X.}},
\oauthor{\bsnm{Gu}, \binits{L.}},
\oauthor{\bsnm{Chen}, \binits{J.}},
\oauthor{\bsnm{Chu}, \binits{N.}},
\oauthor{\bsnm{Wang}, \binits{L.S.}},
\oauthor{\bsnm{Liang}, \binits{Y.}},
\oauthor{\bsnm{Yu}, \binits{J.}}:
Optimal harvesting for nonlinear size-structured populations with nonlocal environmental feedback.
Mathematics
(2026)
\end{botherref}
\endbibitem

\bibitem[\protect\citeauthoryear{Verchota}{1984}]{verchota1984layer}
\begin{barticle}
\bauthor{\bsnm{Verchota}, \binits{G.}}:
\batitle{Layer potentials and regularity for the {D}irichlet problem for {L}aplace's equation in {L}ipschitz domains}.
\bjtitle{Journal of functional analysis}
\bvolume{59}(\bissue{3}),
\bfpage{572}--\blpage{611}
(\byear{1984})
\end{barticle}
\endbibitem

\bibitem[\protect\citeauthoryear{Costabel}{1988}]{costabel1988boundary}
\begin{barticle}
\bauthor{\bsnm{Costabel}, \binits{M.}}:
\batitle{Boundary integral operators on {L}ipschitz domains: elementary results}.
\bjtitle{SIAM journal on Mathematical Analysis}
\bvolume{19}(\bissue{3}),
\bfpage{613}--\blpage{626}
(\byear{1988})
\end{barticle}
\endbibitem

\bibitem[\protect\citeauthoryear{Mitrea}{1997}]{mitrea1997method}
\begin{barticle}
\bauthor{\bsnm{Mitrea}, \binits{D.}}:
\batitle{The method of layer potentials for non-smooth domains with arbitrary topology}.
\bjtitle{Integral Equations and Operator Theory}
\bvolume{29}(\bissue{3}),
\bfpage{320}--\blpage{338}
(\byear{1997})
\end{barticle}
\endbibitem

\bibitem[\protect\citeauthoryear{Wang et~al.}{2026}]{wang2026breakdown}
\begin{barticle}
\bauthor{\bsnm{Wang}, \binits{L.S.}},
\bauthor{\bsnm{Yu}, \binits{J.}},
\bauthor{\bsnm{Liang}, \binits{Y.}},
\bauthor{\bsnm{Zhang}, \binits{J.}}:
\batitle{The breakdown of linear quasi-cycles: Demographic noise and absorbing boundaries in finite predator--prey systems}.
\bjtitle{Electronic Research Archive}
\bvolume{34}(\bissue{6}),
\bfpage{4248}--\blpage{4289}
(\byear{2026})
\end{barticle}
\endbibitem

\bibitem[\protect\citeauthoryear{Mitrea and Mitrea}{2013}]{mitrea2013multi}
\begin{bbook}
\bauthor{\bsnm{Mitrea}, \binits{I.}},
\bauthor{\bsnm{Mitrea}, \binits{M.}}:
\bbtitle{Multi-layer Potentials and Boundary Problems: for Higher-order Elliptic Systems in {L}ipschitz Domains}.
\bpublisher{Springer},
\blocation{Berlin Heidelberg}
(\byear{2013})
\end{bbook}
\endbibitem

\bibitem[\protect\citeauthoryear{Fabes et~al.}{1977}]{fabes1977double}
\begin{barticle}
\bauthor{\bsnm{Fabes}, \binits{E.}},
\bauthor{\bsnm{JODEIT}, \binits{M.}},
\bauthor{\bsnm{Lewis}, \binits{J.E.}}:
\batitle{Double layer potentials for domains with corners and edges}.
\bjtitle{Indiana university mathematics journal}
\bvolume{26}(\bissue{1}),
\bfpage{95}--\blpage{114}
(\byear{1977})
\end{barticle}
\endbibitem

\bibitem[\protect\citeauthoryear{Costabel and Stephan}{1983}]{costabel1983normal}
\begin{barticle}
\bauthor{\bsnm{Costabel}, \binits{M.}},
\bauthor{\bsnm{Stephan}, \binits{M.}}:
\batitle{The normal dervative of the double layer potential on polygons and {G}alerkin approximation}.
\bjtitle{Applicable Analysis}
\bvolume{16}(\bissue{3}),
\bfpage{205}--\blpage{228}
(\byear{1983})
\end{barticle}
\endbibitem

\bibitem[\protect\citeauthoryear{Yu and Wang}{2026}]{yu2026beyond}
\begin{barticle}
\bauthor{\bsnm{Yu}, \binits{J.}},
\bauthor{\bsnm{Wang}, \binits{L.S.}}:
\batitle{Beyond diagonal noise: A better predator-prey modeling framework with cross-covariance}.
\bjtitle{PLoS One}
\bvolume{21}(\bissue{5}),
\bfpage{0350127}
(\byear{2026})
\end{barticle}
\endbibitem

\bibitem[\protect\citeauthoryear{Maz’ya}{1991}]{maz1991boundary}
\begin{bchapter}
\bauthor{\bsnm{Maz’ya}, \binits{V.G.}}:
\bctitle{Boundary integral equations}.
In: \bbtitle{Analysis IV: Linear and Boundary Integral Equations},
pp. \bfpage{127}--\blpage{222}.
\bpublisher{Springer},
\blocation{Berlin Heidelberg}
(\byear{1991})
\end{bchapter}
\endbibitem

\bibitem[\protect\citeauthoryear{Qiao and Nistor}{2012}]{qiao2012single}
\begin{barticle}
\bauthor{\bsnm{Qiao}, \binits{Y.}},
\bauthor{\bsnm{Nistor}, \binits{V.}}:
\batitle{Single and double layer potentials on domains with conical points {I}: {S}traight cones}.
\bjtitle{Integral Equations and Operator Theory}
\bvolume{72}(\bissue{3}),
\bfpage{419}--\blpage{448}
(\byear{2012})
\end{barticle}
\endbibitem

\bibitem[\protect\citeauthoryear{Yu et~al.}{2026}]{yu2026rigorous}
\begin{botherref}
\oauthor{\bsnm{Yu}, \binits{J.}},
\oauthor{\bsnm{Wang}, \binits{L.S.}},
\oauthor{\bsnm{Liang}, \binits{Y.}}:
Rigorous analysis of a nonlocal transport--renewal system for physiologically structured populations.
Mathematical Methods in the Applied Sciences
(2026)
\end{botherref}
\endbibitem

\bibitem[\protect\citeauthoryear{Qiao and Li}{2018}]{qiao2018double}
\begin{barticle}
\bauthor{\bsnm{Qiao}, \binits{Y.}},
\bauthor{\bsnm{Li}, \binits{H.}}:
\batitle{Double layer potentials on polygons and pseudodifferential operators on {L}ie groupoids}.
\bjtitle{Integral Equations and Operator Theory}
\bvolume{90}(\bissue{2}),
\bfpage{14}
(\byear{2018})
\end{barticle}
\endbibitem

\bibitem[\protect\citeauthoryear{Kress et~al.}{1989}]{kress1989linear}
\begin{bbook}
\bauthor{\bsnm{Kress}, \binits{R.}},
\bauthor{\bsnm{Maz'ya}, \binits{V.}},
\bauthor{\bsnm{Kozlov}, \binits{V.}}:
\bbtitle{Linear Integral Equations}
vol. \bseriesno{82}.
\bpublisher{Springer},
\blocation{Berlin Heidelberg}
(\byear{1989})
\end{bbook}
\endbibitem

\bibitem[\protect\citeauthoryear{Kress}{1991}]{kress1991boundary}
\begin{barticle}
\bauthor{\bsnm{Kress}, \binits{R.}}:
\batitle{Boundary integral equations in time-harmonic acoustic scattering}.
\bjtitle{Mathematical and Computer Modelling}
\bvolume{15}(\bissue{3-5}),
\bfpage{229}--\blpage{243}
(\byear{1991})
\end{barticle}
\endbibitem

\bibitem[\protect\citeauthoryear{Grisvard}{2011}]{grisvard2011elliptic}
\begin{bbook}
\bauthor{\bsnm{Grisvard}, \binits{P.}}:
\bbtitle{Elliptic Problems in Nonsmooth Domains}.
\bpublisher{SIAM},
\blocation{Philadelphia, PA}
(\byear{2011})
\end{bbook}
\endbibitem

\bibitem[\protect\citeauthoryear{Kozlov et~al.}{1997}]{kozlov1997elliptic}
\begin{bbook}
\bauthor{\bsnm{Kozlov}, \binits{V.}},
\bauthor{\bsnm{Kozlov}, \binits{V.}},
\bauthor{\bsnm{Maz{\'i}a}, \binits{V.G.}},
\bauthor{\bsnm{Rossmann}, \binits{J.}}:
\bbtitle{Elliptic Boundary Value Problems in Domains with Point Singularities}
vol. \bseriesno{52}.
\bpublisher{American Mathematical Soc.},
\blocation{Providence, RI}
(\byear{1997})
\end{bbook}
\endbibitem

\bibitem[\protect\citeauthoryear{Dauge}{2006}]{dauge2006elliptic}
\begin{bbook}
\bauthor{\bsnm{Dauge}, \binits{M.}}:
\bbtitle{Elliptic Boundary Value Problems on Corner Domains: Smoothness and Asymptotics of Solutions}.
\bpublisher{Springer},
\blocation{Berlin Heidelberg}
(\byear{2006})
\end{bbook}
\endbibitem

\bibitem[\protect\citeauthoryear{Costabel et~al.}{2003}]{costabel2003asymptotics}
\begin{barticle}
\bauthor{\bsnm{Costabel}, \binits{M.}},
\bauthor{\bsnm{Dauge}, \binits{M.}},
\bauthor{\bsnm{Duduchava}, \binits{R.}}:
\batitle{Asymptotics {W}ithout {L}ogarithmic {T}erms for {C}rack {P}roblems}.
\bjtitle{Communications in Partial Differential Equations}
\bvolume{28}(\bissue{5-6}),
\bfpage{869}--\blpage{926}
(\byear{2003})
\end{barticle}
\endbibitem

\bibitem[\protect\citeauthoryear{Schulze}{1990}]{schulze1990mellin}
\begin{bchapter}
\bauthor{\bsnm{Schulze}, \binits{B.-W.}}:
\bctitle{The mellin pseudo-differential calculus on manifolds with corners}.
In: \bbtitle{Symposium “Analysis on Manifolds with Singularities”, Breitenbrunn 1990},
pp. \bfpage{208}--\blpage{289}
(\byear{1990}).
\bcomment{Springer}
\end{bchapter}
\endbibitem

\bibitem[\protect\citeauthoryear{Buffa and Sauter}{2006}]{buffa2006acoustic}
\begin{barticle}
\bauthor{\bsnm{Buffa}, \binits{A.}},
\bauthor{\bsnm{Sauter}, \binits{S.}}:
\batitle{On the acoustic single layer potential: stabilization and {F}ourier analysis}.
\bjtitle{SIAM Journal on Scientific Computing}
\bvolume{28}(\bissue{5}),
\bfpage{1974}--\blpage{1999}
(\byear{2006})
\end{barticle}
\endbibitem

\bibitem[\protect\citeauthoryear{Ammari et~al.}{2009}]{ammari2009layer}
\begin{bbook}
\bauthor{\bsnm{Ammari}, \binits{H.}},
\bauthor{\bsnm{Kang}, \binits{H.}},
\bauthor{\bsnm{Lee}, \binits{H.}}:
\bbtitle{Layer {P}otential {T}echniques in {S}pectral {A}nalysis}.
\bsertitle{Mathematical Surveys and Monographs},
vol. \bseriesno{153}.
\bpublisher{American Mathematical Society},
\blocation{Providence, RI}
(\byear{2009})
\end{bbook}
\endbibitem

\bibitem[\protect\citeauthoryear{Gao}{1991}]{gao1991layer}
\begin{barticle}
\bauthor{\bsnm{Gao}, \binits{W.}}:
\batitle{Layer potentials and boundary value problems for elliptic systems in {L}ipschitz domains}.
\bjtitle{Journal of Functional Analysis}
\bvolume{95}(\bissue{2}),
\bfpage{377}--\blpage{399}
(\byear{1991})
\end{barticle}
\endbibitem

\bibitem[\protect\citeauthoryear{Kenig and Shen}{2011}]{kenig2011layer}
\begin{barticle}
\bauthor{\bsnm{Kenig}, \binits{C.}},
\bauthor{\bsnm{Shen}, \binits{Z.}}:
\batitle{Layer potential methods for elliptic homogenization problems}.
\bjtitle{Communications on pure and applied mathematics}
\bvolume{64}(\bissue{1}),
\bfpage{1}--\blpage{44}
(\byear{2011})
\end{barticle}
\endbibitem

\end{thebibliography}
\end{document}